\documentclass[a4paper,11pt,oneside]{article} 
\usepackage{amsfonts,amssymb}
\usepackage{color} \usepackage{mathrsfs} \let\mathcal\mathscr
\usepackage[all,ps,cmtip]{xy}
 \usepackage[top=1in, bottom=1in, left=1in, right=1in]{geometry}

\title{\sc Hodge numbers and Hodge structures for Calabi-Yau categories of dimension three}
\author{\sc Roland Abuaf \footnote{A. Institut de Math\'ematiques, Rectorat de Paris, 47 rue des \'Ecoles, 75005 Paris. Email : \textit{rabuaf@gmail.com, Roland.Abuaf@ac-paris.fr}}}

\usepackage[dvips]{graphicx}

\usepackage{amsfonts,amssymb}
\usepackage{color} \usepackage{mathrsfs} \let\mathcal\mathscr
\usepackage{mathenv}
\usepackage[all,ps,cmtip]{xy}

\DeclareGraphicsExtensions{eps}
\DeclareGraphicsExtensions{pdf}

\usepackage[T1]{fontenc}
\usepackage{amssymb}

\usepackage{amsmath}
\usepackage{latexsym}

\usepackage[latin1]{inputenc}

\newtheorem{theo}{Theorem}[subsection]

\newtheorem{hypo}[theo]{Hypothesis}
\newtheorem{exem}[theo]{Example}
\newtheorem{rem}[theo]{Remark}

\newtheorem{prop}[theo]{Proposition}
\newtheorem{quest}[theo]{Question}

\newtheorem{defi}[theo]{Definition}

\newtheorem{cor}[theo]{Corollary}

\def\DB{\mathrm{D^{b}}}
\def\OO{\mathcal{O}}

\def\R0{\mathrm{R^{0}}}

\def\HH{\mathrm{HH}}

\def\HHH{\mathrm{Hom}}

\def\OO{\mathcal{O}}

\def\D{\mathcal{D}}

\def\E{\mathcal{E}}

\def\C{\mathcal{C}}
\def\cc{\mathrm{c}}
\def\F{\mathcal{F}}
\def\GG{\mathrm{G}}
\def\G{\mathcal{G}}

\def\A{\mathcal{A}}

\def\X{\mathcal{X}}

\def\ds{\displaystyle}

\newcommand{\hooklongrightarrow}{\lhook\joinrel\longrightarrow}
\newcommand{\leftexp}[2]{{\vphantom{#2}}^{#1}{#2}}

\newenvironment{proof}
{
\noindent
\textit{\underline{Proof}} :\\
$\blacktriangleright\;$%
}
{\hspace{\stretch{1}}%
$\blacktriangleleft$}

\begin{document}

\maketitle

\begin{abstract}
Let $\A$ be a smooth proper $\mathbb{C}$-linear triangulated category Calabi-Yau of dimension $3$ endowed with a (non-trivial) rank function. Using the homological unit of $\A$ with respect to the given rank function, we define Hodge numbers for $\A$. 

If the classes of unitary objects generate the complexified numerical $\mathrm{K}$-theory of $\A$ (hypothesis satisfied for many examples of smooth proper Calabi-Yau categories of dimension $3$), it is proved that these numbers are independent of the chosen rank function : they are intrinsic invariants of the triangulated category $\A$.

In the special case where $\A$ is a semi-orthogonal component of the derived category of a smooth complex projective variety and the homological unit of $\A$ is $\mathbb{C} \oplus \mathbb{C}[3]$ (that is $\A$ is strict Calabi-Yau with respect to the rank function), we define a Hodge structure on the Hochschild homology of $\A$. The dimensions of the Hodge spaces of this structure are the Hodge numbers aforementioned.

Finally, we give some numerical applications toward the Homological Mirror Symmetry conjecture for cubic sevenfolds and double quartic fivefolds.
\end{abstract}

\vspace{\stretch{1}}

\newpage
\tableofcontents
\begin{section}{Introduction}

As part of his Homological Mirror Symmetry conjecture (see \cite{Kontsevich}) M. Kontsevich predicted that the Hochschild homology (or rather the cyclic homology) of a smooth proper triangulated category should be endowed with a Hodge structure. It has been suggested that this Hodge structure could be obtained via the degeneration of a certain spectral sequence, \`a la Deligne-Illusie (\cite{KKP, Kaledin, KKP2}). This approach looks however rather abstract and it seems difficult to use it in practice to compute Hodge numbers for a given example of Calabi-Yau category.

\bigskip

In this paper, we use the more concrete (though less functorial, at first sight) theory of homological units (see \cite{homounit}) in order to define (and compute) Hodge numbers for smooth proper $\mathbb{C}$-linear triangulated category which are Calabi-Yau of dimension $3$. Our main definition and our main result are (see definition \ref{defihodgenumbers} and Theorem \ref{invahodgenumbers}):

\begin{defi} \label{defihodgenumbers}
Let $\A$ be a smooth proper triangulated category which is Calabi-Yau of dimension $3$ and endowed with a non-trivial rank function. Let $\mathfrak{T}_{\A}^{\bullet}$ be a homological unit for $\A$ with respect to the rank function. We define the \textit{Hodge numbers} of $\A$ as:
\begin{enumerate}
\item for all $i \in [0, \ldots, 3]$, $h^{i,0}(\A) = \mathfrak{T}_{\A}^{3-i}$,
\item $h^{3,1}(\A) = \dim \HH_{-2}(\A) - h^{2,0}(\A)$,
\item $h^{3,2}(\A) = h^{1,0}(\A)$ and $h^{2,1}(\A) = \dim \HH_{-1}(\A) - h^{1,0}(\A) - h^{3,2}(\A)$,
\item $h^{3,3}(\A) = h^{0,0}(\A)$ and $h^{1,1}(\A) = h^{2,2}(\A) = \dfrac{\dim \HH_0(\A) - h^{0,0}(\A) - h^{3,3}(\A)}{2}$.
\item $h^{p,q}(\A) = h^{q,p}(\A)$ for any $p,q \in [0, \ldots, 3]$.
\end{enumerate}

\end{defi}

\begin{theo} 
Let $\A$ be a smooth proper triangulated category which is Calabi-Yau of dimension $3$. Let $\mathrm{rank}_1, \mathrm{rank}_2$ be a non-trivial numerical rank functions on $\A$ and let $\mathfrak{T}_{\A,1}^{\bullet}, \mathfrak{T}_{\A,2}^{\bullet}$ be homological units for $\A$ with respect to $\mathrm{rank}_1$, $\mathrm{rank}_2$. Let $\mathrm{cl} : \A \longrightarrow \mathrm{K}_{num}(\A)$ be the class map and denote by $\A_{unitary}^{(i)}$ the set of objects $\F \in \A$ such that $\HHH^{\bullet}_{\A}(\F,\F) \simeq \mathfrak{T}_{\A,i}^{\bullet}$ as graded rings. Finally, for all $p,q \in [0,\ldots,3]$, we denote by $h^{p,q}_i(\A)$ the Hodge numbers of $\A$ associated to $\mathfrak{T}_{\A,i}^{\bullet}$ as in definition \ref{defihodgenumbers}. We have the following:

\begin{enumerate}
\item  If both $\mathrm{cl}(\A_{unitary}^{(1)})$ and $\mathrm{cl}(\A_{unitary}^{(2)})$ generate $\mathrm{K}_{num}(\A) \otimes \mathbb{C}$, then:

\[ h^{p,q}_1(\A) = h^{p,q}_2(\A),\]
for all $p,q \in [0,\ldots,3]$.

\item If $\mathfrak{T}_{\A,1}^{\bullet} = \mathbb{C} \oplus \mathbb{C}[3]$, $\mathrm{cl}(\A_{unitary}^{(1)})$ generates $\mathrm{K}_{num}(\A) \otimes \mathbb{C}$ and there exists a unitary object in $\A$ with respect to $\mathfrak{T}_{\A,2}^{\bullet}$, then:

\[ h^{p,q}_1(\A) = h^{p,q}_2(\A),\]
for all $p,q \in [0,\ldots,3]$.
\end{enumerate}
\end{theo}
The key hypothesis in the above Theorem, namely that the classes of unitary objects generate the complexified numerical $\mathrm{K}$-theory of $\A$ is satisfied for many examples of smooth proper Calabi-Yau categories of dimension $3$. We shall describe such examples in section $2.2$ of the paper. 

\bigskip

\noindent The plan of the paper is the following:
\begin{itemize}
\item In section $2.1$, we recall the definitions of rank functions, homological units and their basic properties. We compute the units for many examples of smooth proper Calabi-Yau category of dimension $3$.
\item In section $2.2$, we study the invariance of the homological unit with respect to the chosen rank function.
\item In section $3.1$, we define the Hodge numbers for a smooth proper Calabi-Yau category endowed with a non-trivial rank function. We compute them for the examples introduced in section $1.1$.
\item In section $3.2$, we assume that our triangulated category is a semi-orthogonal component of the derived category of coherent sheaves on a smooth complex projective variety and that its homological unit (with respect to the rank function coming from the variety) is $\mathbb{C} \oplus \mathbb{C}[3]$. We then define a Hodge structure on the Hochschild homology of this category. The dimensions of the corresponding Hodge spaces are the Hodge numbers aforementioned.
\item In section $3.3$, we give some numerical applications toward the Homological Mirror Symmetry conjecture for cubic sevenfolds and double quartic fivefolds.
\end{itemize} 

\bigskip

\noindent \textbf{Acknowledgment :} I am very grateful to Matt B. Young for stimulating discussions on Mirror Symmetry for cubic sevenfolds and double quartic fivefolds.

\bigskip

\noindent \textbf{Conventions :} We work over the field of complex numbers. We only consider $\mathbb{C}$-linear triangulated categories which can be realized as derived categories of $DG$-modules over a proper $DG$-algebra (call them \textit{proper}). We will mostly consider triangulated categories which can be realized as derived categories of $DG$-modules over a smooth $DG$-algebra (call them \textit{smooth}).

\end{section}

\begin{section}{Homological units and invariance properties for Calabi-Yau categories of dimension three} 
In this section, we recall the definition of homological unit \cite{homounit} and prove some invariance properties for the homological units of Calabi-Yau categories of dimension $3$.

\begin{subsection}{Calabi-Yau categories and homological units}
 We first recall some definitions related to such categories.

\begin{defi}
Let $\A$ be a triangulated category. We say that $\A$ is a \textbf{Calabi-Yau category of dimension $p$} if the shift by $[p]$ is a Serre functor for $\A$. We say furthermore that $\A$ is a \textbf{geometric Calabi-Yau category of dimension $p$} if there exists a smooth projective variety $X$ and a semi-orthogonal decomposition

\[ \DB(X) = \langle \A, \leftexp{\perp}{\A} \rangle, \]
such the shift by $[p]$ is a Serre functor for $\A$.
\end{defi}
This definition already appeared many times in the literature. It has been studied in details when $X$ is a cubic fourfold and $p=2$ in \cite{kuz1} and more generally when $X$ is a hypersurface in (or a double cover of) a rational homogeneous space for any $p$ in \cite{kuz2}. The study of the $K3$ category appearing in a cubic fourfold in connection with rationality problems and the Hodge theory of cubic fourfolds has been carried out in many papers (with starting point \cite{kuz1}). Let us quote \cite{huy1}, which provides detailed study of recent works on the subject.

\bigskip

The notion of homological units has been introduced in \cite{homounit} as a categorical substitute for the algebra $H^{\bullet}(\OO_{\X})$. It is useful when the category under study is not necessarily the derived category of a smooth projective Deligne-Mumford stack. We recall the definition of homological units in the context of triangulated categories.

\begin{defi}
Let $\A$ be a smooth proper triangulated category. A \textbf{rank function} on $\A$ is a function $\mathrm{rk} : \A \longrightarrow \mathbb{Z}$ which is additive with respect to exact triangles and such that $\mathrm{rk}(\F[1]) = - \mathrm{rk}(\F)$, for any $\F \in \A$. We say that the rank function is trivial if it is the zero function and we say that it is a \textbf{numerical rank function} if it descends to a map:
\[ \mathrm{rk} : \mathrm{K}_{num}(\A) \longrightarrow \mathbb{Z},\]
where $\mathrm{K}_{num}(\A)$ is the quotient of the $\mathrm{K}$-theory of $\A$ by the kernel of the bilinear form:
\begin{equation*}
\begin{split} 
\chi : \mathrm{K}_0(\A) \times \mathrm{K}_0(\A) & \longrightarrow \mathbb{Z}\\
(\E, \F) & \longmapsto \sum_{k \geq 0} (-1)^k \dim \mathrm{Ext}^k(\E,\F).
\end{split}
\end{equation*}

\end{defi}
Let $X$ be a smooth projective variety. The Grothendieck-Riemann-Roch Theorem shows that the rank of a numerically trivial object in $\mathrm{K}_{0}(X)$ is necessarily $0$. Hence, the natural rank function on $\DB(X)$ is a numerical rank function. Let $\A$ be a semi-orthogonal component of $\DB(X)$, then we have decompositions:
\[ \mathrm{K}_{0}(X) = \mathrm{K}_{0}(\A) \oplus \mathrm{K}_{0}(\leftexp{\perp}{\A}) \]
and
\[ \mathrm{K}_{num}(X) = \mathrm{K}_{num}(\A) \oplus \mathrm{K}_{num}(\leftexp{\perp}{\A}). \]
As a consequence, the natural rank function on $\DB(X)$ restricts to a numerical rank function on $\A$.

\begin{defi} \label{homounit}
Let $\A$ be a triangulated category endowed with a non-trivial rank function. A graded algebra $\mathfrak{T}_{\A}^{\bullet}$ is called a \textbf{homological unit} for $\A$, if $\mathfrak{T}_{\A}^{\bullet}$ is maximal for the following property. For any object $\F \in \A$, there exists a pair of morphisms $i_{\F} : \mathfrak{T}_{\A}^{\bullet} \longrightarrow  \HHH^{\bullet}_{\A}(\F,\F)$ and $t_{\F} : \HHH^{\bullet}_{\A}(\F, \F) \longrightarrow \mathfrak{T}_{\A}^{\bullet}$ such that:
 \begin{itemize}

 \item the morphism $i_{\F} : \mathfrak{T}_{\A}^{\bullet} \longrightarrow  \HHH^{\bullet}_{\A}(\F,\F)$ is a graded algebra morphism which is functorial in the following sense. Let $\F, \G \in \A$ and let $a \in \mathfrak{T}_{\A}^{k}$ for some $k$. Then, for any morphism $\psi : \F \longrightarrow \G$, there is a commutative diagram:
\begin{equation*}
\xymatrix{ \F \ar[rr]^{i_{\F}} \ar[dd]^{\psi} & &\F [k] \ar[dd]^{\psi[k]} \\
& &  \\
\G \ar[rr]^{i_{\G}} & & \G [k]}
\end{equation*}

\item the morphism $t_{\F} : \HHH^{\bullet}_{\A}(\F, \F) \longrightarrow \mathfrak{T}_{\A}^{\bullet}$ is a graded vector spaces morphism such that for any $\F \in \A$ and any $a \in \mathfrak{T}_{\A}^{\bullet}$, we have $t_{\F}(i_{\F}(a)) = \mathrm{rank}(\F).a$.

\end{itemize}

With hypotheses as above, an object $\F \in \A$ is said to be \textup{unitary}, if $\HHH^{\bullet}_{\A}(\F,\F) \simeq \mathfrak{T}_{\A}^{\bullet}$ as graded rings, where $\mathfrak{T}_{\A}^{\bullet}$ is a homological unit for $\A$.

\end{defi}

\begin{rem} \label{rem1}
\upshape{
\begin{enumerate}
\item In \cite{homounit}, the notion of rank function was only defined on <<nice>> abelian categories (and we would descend them on the derived categories of such abelian categories). It seems that there is no harm in defining them directly at the level of triangulated categories. This is the approach we have chosen for the present paper.

\item Let $\X$ be a projective Deligne-Mumford stack which can be written as a global quotient $[X/ \GG]$ where $X$ is a projective variety and $\GG$ is a reductive group acting linearly on $X$. Let $\OO_{X}(1)$ be a $\GG$-equivariant line bundle. A minor modification of the arguments in Theorem 4 of \cite{orlov} shows that there is an equivalence:

\[ D^{perf}(\X) \simeq D^{perf}(\C), \]
where

\[\C = \mathrm{RHom}^{\GG}_X \left(\bigoplus_{i=0}^{\dim X} \OO_{X}(i), \bigoplus_{i=0}^{\dim X} \OO_{X}(i) \right).\]

Let us consider the rank of an $\OO_{\X}$-module as a rank function on $D^{perf}(\X)$. In such a case, we have $\mathfrak{T}_{D^{perf}(\X)}^{\bullet} = H^{\bullet}(\OO_{\X})$. Furthermore, for any $\F \in \D^{perf}(\X)$, the morphism $i_{\F}$ is the tensor product (over $\OO_{\X}$) with the identity map of $\F$ and the morphism $t_{\F}$ is the trace map $\mathrm{Hom}_{\X}^{\bullet}(\F,\F) \longrightarrow H^{\bullet}(\OO_{\X})$. 

\item In the above definition, the existence of the morphisms $i_{\F}$ and $t_{\F}$ for all $\F \in \A$ and their functorial properties is equivalent to the existence of a morphism of graded algebras:

\[ \mathfrak{T}_{\A}^{\bullet} \longrightarrow \HH^{\bullet}(\C), \]
where $\HH^{\bullet}(\A)$ is the Hochschild cohomology of $\A$. As the rank function on $\A$ is non-trivial, then the map $\mathfrak{T}_{_A}^{\bullet} \longrightarrow \HH^{\bullet}(\A)$ is injective.

 \item On the other hand, the definition and the (splitting) properties of the morphisms $t_{\F}$, for $\F \in \A$ with non-zero rank do not seem to be easily written using only the notion of graded morphisms between $\HH^{\bullet}(\A)$ and $\mathfrak{T}_{\A}^{\bullet}$. It is appears that there is no obvious way to write that $t_{\F}$ splits $i_{\F}$ whenever the rank of $\F$ is not zero only in terms of Hochschild cohomology.

\item If $\A$ contains a unitary object whose rank is not zero, then the homological unit of $\A$ is necessarily unique (though the embedding of the homological unit in $\HH^{\bullet}(\A)$ is certainly not unique). This follows from the maximality condition imposed in definition \ref{homounit}.

\item Let $X$ and $Y$ be smooth projective varieties of dimension less or equal to $4$ such that $\DB(X) \simeq \DB(Y)$. It is proved in \cite{homounit} that the algebras $H^{\bullet}(\OO_X)$ and $H^{\bullet}(\OO_Y)$ are isomorphic. This suggests that the homological unit of a DG category of geometric origin could be independent of the embedding into the derived category of a smooth projective Deligne-Mumford stack (at least if the dimensions of the varieties are small enough). In the next subsection, we will investigate in more details the invariance properties of homological units attached to geometric Calabi-Yau categories of dimension $3$.
\end{enumerate}
}
\end{rem}

Let $X \subset \mathbb{P}^5$ be a smooth cubic hypersurface. According to \cite{kuz1}, we have a semi-orthogonal decomposition:

\[\DB(X) = \langle \A_{X}, \OO_{X}, \OO_{X}(1), \OO_{X}(2) \rangle, \]
where the Serre functor of $\A_X$ is the twist by $[2]$. This category has been studied in some details, most prominently in connection with the rationality problem for cubic fourfolds (see \cite{kuz1, richard, huy1} for instance). 

\begin{prop} \label{prophodge} 
We keep notations as above. Then:
\begin{enumerate}
\item The homological unit of $\A_X$ with respect to the rank function coming from $\DB(X)$ is $\mathbb{C} \oplus \mathbb{C}[2]$.
\item The homological unit of $\A_X$ allows to define a weight-2 Hodge structure on $\mathrm{HH}_{\bullet}(\A_X)$.
\end{enumerate}
\end{prop}

\begin{proof}
We first prove that the homological unit of $\A_X$ with respect to the rank function coming from $\DB(X)$ is $\mathbb{C} \oplus \mathbb{C}[2]$. The twisted version of the Hochschild-Kostant-Rosenberg provides (see \cite{hochschild} for instance) isomorphisms:

\begin{equation*}
\begin{split}
\tau_{\mathrm{HH}^{\bullet}} : &\ \ \mathrm{HH}^{\bullet}(\DB(X)) \simeq \bigoplus_{p+q = \bullet}H^{q}(X, \bigwedge^p T_{X}) \\
\tau_{\mathrm{HH}_{\bullet}} : &\ \ \mathrm{HH}_{\bullet}(\DB(X)) \simeq \bigoplus_{p-q = \bullet}H^{q}(X, \Omega^p_X), \\
\end{split}
\end{equation*}
where $\tau_{\mathrm{HH}^{\bullet}}$ is compatible with ring structures on both sides and $\tau_{\mathrm{HH}_{\bullet}}$ is compatible with the module structures over $\mathrm{HH}^{\bullet}(\DB(X)) \simeq \bigoplus_{p+q = \bullet}(H^{q}(X, \bigwedge^p T_{X})$ on both sides (the module structure on the left is given by the composition of morphisms while the module structure on the right is given by contraction of forms). Since the Hodge diamond of a smooth cubic fourfold is:
\begin{equation*}
\begin{tabular}{ccccccccc} 
& & & & 1 & & & & \\
& & & 0& &0 & & &  \\
& & 0& &1 & &0 & & \\
&0& &0 & &0 & &0& \\
0& &1 & &21 & &1 & &0\\
&0& &0 & &0 & &0& \\
& & 0& &1 & &0 & & \\
& & & 0& &0 & & &  \\
& & & & 1 & & & & \\
\end{tabular}
\end{equation*}
we deduce that $\HH_{-4}(\DB(X)) = \HH_{4}(\DB(X)) = \HH_{-3}(\DB(X)) = \HH_{3}(\DB(X)) = \HH_{-1}(\DB(X)) = \HH_{1}(\DB(X)) =0$, $\HH_{-2}(\DB(X)) \simeq \HH_{2}(\DB(X)) \simeq \mathbb{C}$ and $\HH_{0}(\DB(X)) \simeq \mathbb{C}^{25}$. On the other hand, according to \cite{kuz3}, we have a splitting:

\[ \HH_{\bullet}(\DB(X)) = \HH_{\bullet}(\A_X) \oplus \HH_{\bullet}(\langle \OO_X \rangle) \oplus  \HH_{\bullet}(\langle \OO_X(1) \rangle)   \oplus \HH_{\bullet}(\langle \OO_X(2) \rangle). \]
Hence, we obtain the following isomorphisms:
\begin{equation*}
\begin{split}
& \HH_{-4}(\A_X) = \HH_{4}(\A_X) = \HH_{-3}(\A_X) = \HH_{3}(\A_X) = \HH_{-1}(\A_X) = \HH_{1}(\A_X) =0\\
& \HH_{-2}(\A_X) \simeq \HH_{2}(\A_X) \simeq \mathbb{C}\\
& \HH_{0}(\A_X) \simeq \mathbb{C}^{22}.
\end{split}
\end{equation*}
 Since the shift by $[2]$ is a Serre functor for $\A_X$, we get the isomorphism of graded vector spaces:
 \[ \HH_{\bullet}(\A_X) \simeq \HH^{\bullet - 2}(\A_X).\] 
 We can therefore compute the dimension of the Hochschild cohomology spaces of $\A_X$:

\begin{equation*}
\begin{split}
& \HH^1(\A_X) = \HH^3(\A_X) =0\\
& \HH^0(\A_X) \simeq \HH^{4}(\A_X) \simeq \mathbb{C}\\
& \HH^{2}(\A_X) \simeq \mathbb{C}^{22}.
\end{split}
\end{equation*}

\bigskip

Let $l \subset X$ be a line in $X$ and let $\mathcal{J}_l \subset \OO_X$ be the ideal sheaf of $l$. It is easily proved that $H^0(X, \mathcal{J}_l(1)) = \mathbb{C}^4$ and that $\mathcal{J}_l(1)$ is generated by global sections. As a consequence, we have an exact sequence:

\[ 0 \longrightarrow \mathcal{F}_l \longrightarrow H^0(X, \mathcal{J}_l(1)) \otimes \OO_X \longrightarrow \mathcal{J}_l(1) \longrightarrow 0, \]
where $\mathcal{F}_l$ is the kernel of the evaluation map $H^0(X, \mathcal{J}_l(1)) \otimes \OO_X \longrightarrow \mathcal{J}_l(1)$. The bundle $\mathcal{F}_l$ is usually called the \textit{syzygy bundle associated to $l$}. It is easily proved that:

\[H^{\bullet}(X, \mathcal{F}_l) = H^{\bullet}(X, \mathcal{F}_l(-1)) = H^{\bullet}(X, \mathcal{F}_l(-2)) = 0, \]
so that $\mathcal{F}_l \in \A_X$. This bundle appeared already many times in the literature. It was used for instance in \cite{kuzmarku} to prove that the variety of lines in $X$ is a connected component of the moduli space of sheaves in $\A_X$ having the same Chern classes as $\mathcal{F}_l$. The sheaf $\F_l$ has rank $3$, is reflexive and we have $c_1(\F_l) = \OO_X(-1)$. Since $H^0(X, \F_l) = 0$ and $H^0(X, \F_l^{\vee}(-1) = H^0(X, \F_l(-2)) = 0$, we deduce that $\F_l$ is a stable sheaf with respect to $\OO_X(1)$. In particular, we have:

\begin{equation*}
\mathrm{Hom}(\F_l, \F_l) = \mathbb{C}.
\end{equation*}
Since $\F_l \in \A_X$ and $\A_X$ is Calabi-Yau category of dimension $2$, we find $\mathrm{Ext}^2(\F_l \F_l) = \mathbb{C}$ and $\mathrm{Ext}^{p}(\F_l, \F_l) = 0$ for $p \geq 3$. Let $\mathfrak{T}^{\bullet}_{\A_X}$ be a homological unit for $\A_X$ with respect to the embedding in $\DB(X)$. As $\mathrm{rank}(\F_l) = 3$ and $\F_l \in \A_X$, we have embedding of graded algebras:

\[ \mathfrak{T}^{\bullet}_{\A_X} \hooklongrightarrow \mathrm{Ext}^{\bullet}(\F_l, \F_l),\]
by definition of an homological unit for $\A_X$. We therefore find $\mathfrak{T}^0_{\A_X} \simeq \mathfrak{T}^2_{\A_X} \simeq \mathbb{C}$ and $\mathfrak{T}^p_{\A_X} = 0$ for $p \geq 3$ or $p <0$. By item $2$ of remark \ref{rem1} above, we have an embedding:

\begin{equation*}
\mathfrak{T}^{\bullet}_{\A_X} \hooklongrightarrow \HH^{\bullet}(\A_X).
\end{equation*}
We have already observed that $\HH^{1}(\A_X) = 0$, which implies $\mathfrak{T}^{1}_{\A_X} = 0$. We conclude that if $\mathfrak{T}^{\bullet}(\A_X)$ is a homological unit for $\A_X$ with respect to the rank function coming from $\DB(X)$, then we necessarily have $\mathfrak{T}^{\bullet}_{\A_X} = \mathbb{C} \oplus \mathbb{C}[2]$. We now prove that $\mathbb{C} \oplus \mathbb{C}[2]$ is indeed a homological unit for $\A_X$ with respect to the rank function coming from $\DB(X)$. For any $\E \in \A_X$, any $a \in \mathbb{C}$ and any $f \in \HHH(\E,\E)$, we put:

\[ i^{0}_{\E}(a) = a.\mathrm{id}_{\E} \ \ \textrm{and} \ \ t^{0}_{E}(f) = \mathrm{Trace}(f),\]
where $\mathrm{Trace}$ is the trace map inherited from $\DB(X)$. It is clear that $t^0_{\E}(i^0_{\E}(a)) = \mathrm{rank}(\E).a$, for any $a \in \mathbb{C}$. Furthermore, as $\A_X$ is Calabi-Yau category of dimension $2$, we have functorial isomorphisms:

\[ \mathrm{Ext}^{2}(\E,\E) \simeq \HHH(\E,\E)^*, \]
for all $\E \in \A_X$. As a consequence, the pair of morphisms:

\[ i^{2}_{\E} = \left(t^{0}_{\E} \right)^* : \mathbb{C} \longrightarrow \mathrm{Ext}^{2}(\E,\E), \ \ t^{2}_{\E} = \left(i_{\E}^{2} \right)^{*} : \mathrm{Ext}^{2}(\E,\E) \longrightarrow \mathbb{C} \]
is well-defined and satisfies $t_{\E}^{2} \circ i_{\E}^{2} = \left(t_{E}^{0} \circ i_{E}^{0}\right)^{*} = \mathrm{rank}(E).\mathrm{id}_{\mathbb{C}}$. The algebra structure on $\mathbb{C} \oplus \mathbb{C}[2]$ is trivial. Hence, the functoriality of the isomorphisms $\HHH(E,E)^* \simeq \mathrm{Ext}^{2}(\E,\E)$, the functoriality of the morphisms $i_{\E}$ and the properties of the trace map imply that $\mathbb{C} \oplus \mathbb{C}[2]$ is indeed a homological unit for $\A_X$.

\bigskip

We now the second item of proposition \ref{prophodge}. By \cite{kuz3}, we have a graded decomposition:
\begin{equation*}
\HH_{\bullet}(\DB(X)) = \HH_{\bullet}(\A_X) \oplus \HH_{\bullet}(\langle \OO_X \rangle) \oplus \HH_{\bullet}(\langle \OO_X(1) \rangle) \oplus \HH_{\bullet}(\langle \OO_X(2) \rangle).
\end{equation*}
the Hochschild-Kostant-Rosenberg isomorphism, we have an isomorphism:
\begin{equation*}
\begin{split}
\tau_{\mathrm{HH}_{\bullet}} : &\ \ \mathrm{HH}_{\bullet}(\DB(X)) \simeq \bigoplus_{p-q = \bullet}H^{q}(X, \Omega^p_X), \\
\end{split}
\end{equation*}
Furthermore, by Hodge symmetry, the complex conjugation induces an isomorphism:
\begin{equation*}
\mathrm{HS}_{X} : H^p(X,\Omega^q_X) \simeq H^{q}(X, \Omega^p_X).
\end{equation*}
Composing complex conjugation with the inverse of the map $\tau_{\mathrm{HH}_{\bullet}}$, we find an involution:
\begin{equation*}
\mathrm{c}_X : \mathrm{HH}_{\bullet}(\DB(X)) \simeq \mathrm{HH}_{-\bullet}(\DB(X))
\end{equation*}
It is easily checked that the map $\mathrm{c}_X$ stabilizes $\HH_{\bullet}(\langle \OO_X \rangle) \oplus \HH_{\bullet}(\langle \OO_X(1) \rangle) \oplus \HH_{\bullet}(\langle \OO_X(2) \rangle)$ in the decomposition:
\begin{equation*}
\HH_{\bullet}(\DB(X)) = \HH_{\bullet}(\A_X) \oplus \HH_{\bullet}(\langle \OO_X \rangle) \oplus \HH_{\bullet}(\langle \OO_X(1) \rangle) \oplus \HH_{\bullet}(\langle \OO_X(2) \rangle).
\end{equation*}
In particular, since we have an equality:
\begin{equation*}
\HH_{\bullet}(\A_X) = \HH_{\bullet}(X)/ \left( \HH_{\bullet}(\langle \OO_X \rangle) \oplus \HH_{\bullet}(\langle \OO_X(1) \rangle) \oplus \HH_{\bullet}(\langle \OO_X(2) \rangle) \right)
\end{equation*}
the map $\mathrm{c}_X$ descends to an involution:
\begin{equation*}
\cc_{\A_X} : \HH_{\bullet}(\A_X) \simeq \HH_{- \bullet}(\A_X).
\end{equation*}
Notice that $\cc_{\A_X}$ fixes $\HH_{0}(\A_X)$ as Hodge symmetry fixes $\bigoplus_{p} H^{p,p}(X)$. We have $\mathfrak{T}_{\A_X}^0 \hooklongrightarrow \HH_{-2}(\A_X)$, so that $\cc_{\A_X}(\mathfrak{T}_{\A_X}^0) \hooklongrightarrow \HH_{2}(\A_X)$. Hence we can define:

\begin{equation*}
\begin{split}
& \HH_{\bullet}^{2,0}(\A_X) :=  \mathfrak{T}_{\A_X}^0 \\
& \HH_{\bullet}^{0,2}(\A_X) :=  \cc_{\A_X}(\mathfrak{T}_{\A_X}^0) \\
& \HH_{\bullet}^{1,1}(\A_X) := \HH_{0}(\A_X).\\
\end{split}
\end{equation*}
Since \[ \HH_{-4}(\A_X) = \HH_{4}(\A_X) = \HH_{-3}(\A_X) = \HH_{3}(\A_X) = \HH_{-1}(\A_X) = \HH_{1}(\A_X) =0,\] and
\[ \HH_{2}(\A_X) \simeq \HH_{-2}(\A_X) \simeq \mathbb{C} \]
while $\cc_{\A_X}$ fixes $\HH_{0}(\A_X)$ and $\cc_{\A_X}\left(\HH_{\bullet}^{2,0}(\A_X) \right) = \HH_{\bullet}^{0,2}(\A_X)$, we deduce that there is a direct sum decomposition:
\[ \HH_{\bullet}(\A_X) = \HH_{\bullet}^{2,0}(\A_X) \oplus \HH_{\bullet}^{1,1}(\A_X) \oplus  \HH_{\bullet}^{0,2}(\A_X) \]
which is a weight-2 pure Hodge structure on $\HH_{\bullet}(\A_X)$.

\end{proof}

\begin{rem}
\begin{enumerate}
\item The homological unit of $\A_X$ (with respect to the rank function inherited from $X$) being $\mathbb{C} \oplus \mathbb{C}[2]$ is in sharp contrast with the fact that $H^{\bullet}(\OO_{X}) = \mathbb{C}$. Nevertheless, it matches with the general principle that $\A_X$ should be seen as a <<non-commutative>> $K3$ surface.

\item The topological $K$-theory of $\A_X$  as well as a weight-2 Hodge structure on its complexification have been defined and studied in \cite{richard}. Their definition of the weight-2 structure is however slightly ad-hoc as they rely on the fact that the Hodge diamond of a cubic fourfold is almost equal, in an obvious way, to that of a $K3$ surface. 

\item It is easily checked that weight-$2$ Hodge structure we define on $\A_X$ coincide with that of \cite{richard}. Nevertheless, our definition of the weight structure seems to be more natural as it adapts to many geometric Calabi-Yau categories of dimension less or equal to $3$. We will see in the second section of this paper how the second item of proposition \ref{prophodge} can be generalized for such categories.

\item One can also define a sub-lattice of $\HH_{\bullet}(\A_X)$ as the projection on $\HH_{\bullet}(\A_X)$ of the sub-lattice of $\HH_{\bullet}(\DB(X))$ given by $\tau_{\HH_{\bullet}}^{-1}(\left( \bigoplus_{p-q = \bullet} H^q(X,\Omega^p_X) \cap H^{p+q}(X,\mathbb{Z}) \right)$. The complexification of this lattice in $\HH_{\bullet}(\A_X)$ is equal to $\HH_{\bullet}(\A_X)$.
\end{enumerate}
\end{rem}

\bigskip

In case $\A$ is an admissible Calabi-Yau subcategory of the derived category of a projective Deligne-Mumford stack which contains a spherical object whose rank is non-zero, the homological unit is easily computed:

\begin{prop} \label{prop1}
Let $\X$  be projective Deligne-Mumford stack and $\A \subset D^{perf}(\X)$ be an admissible subcategory. Assume that $\A$ is a Calabi-Yau category of dimension $p$ and that it contains a $p$-spherical object whose rank (as a $\OO_{\X}$-module) is non-zero. Then, the homological unit of $\A$ (with respect to the rank function coming from $D^{perf}(\X)$) is $\mathbb{C} \oplus \mathbb{C}[p]$.
\end{prop}

We recall that an object $\E \in \A$ is $p$-spherical if:
\[ \mathrm{Ext}^{\bullet}(\E, \E) = \mathbb{C} \oplus \mathbb{C}[p].\]

\begin{proof}
Let $\mathfrak{T}_{\A}^{\bullet}$ be a homological unit for $\A$ with respect to the rank function coming from $D^{perf}(\X)$. Let $\E$ be a $p$-spherical object in $\A$ which rank is not zero. By definition of homological unit, we must have $\mathfrak{T}_{\A}^{\bullet} \hooklongrightarrow \mathrm{Ext}^{\bullet}(\E,\E) = \mathbb{C} \oplus \mathbb{C}[p]$.

\bigskip

We now prove that $\mathbb{C} \oplus \mathbb{C}[p]$ is indeed a homological unit for $\A$ with respect to the rank function coming from $D^{perf}(\X)$. For any $\E \in \A$, any $a \in \mathbb{C}$ and any $f \in \HHH(\E,\E)$, we put:

\[ i^{0}_{\E}(a) = a.\mathrm{id}_{\E} \ \ \textrm{and} \ \ t^{0}_{E}(f) = \mathrm{Trace}(f),\]
where $\mathrm{Trace}$ is the trace map inherited from $\DB(\X)$. It is clear that $t^0_{\E}(i^0_{\E}(a)) = \mathrm{rank}(\E).a$, for any $a \in \mathbb{C}$. Furthermore, as $\A_X$ is Calabi-Yau category of dimension $p$, we have functorial isomorphisms:

\[ \mathrm{Ext}^{p}(\E,\E) \simeq \HHH(\E,\E)^*, \]
for all $\E \in \A$. As a consequence, the pair of morphisms:

\[ i^{p}_{\E} = \left(t^{0}_{\E} \right)^* : \mathbb{C} \longrightarrow \mathrm{Ext}^{p}(\E,\E), \ \ t^{p}_{\E} = \left(i_{\E}^{0} \right)^{*} : \mathrm{Ext}^{p}(\E,\E) \longrightarrow \mathbb{C} \]
is well-defined and satisfies $t_{\E}^{p} \circ i_{\E}^{p} = \left(t_{E}^{0} \circ i_{E}^{0}\right)^{*} = \mathrm{rank}(E).\mathrm{id}_{\mathbb{C}}$. The algebra structure on $\mathbb{C} \oplus \mathbb{C}[p]$ is trivial. Hence, the functoriality of the isomorphisms $\HHH(E,E)^* \simeq \mathrm{Ext}^{p}(\E,\E)$, the functoriality of the morphisms $i_{\E}$ and the properties of the trace map imply that $\mathbb{C} \oplus \mathbb{C}[p]$ is indeed a homological unit for $\A$.

\end{proof}

\begin{exem} \label{exem1}
\begin{enumerate}
\upshape{
\item Let $X \subset \mathbb{P}^{8}$ be a generic cubic hypersurface. It was checked in \cite{maniliev} that $X$ is a linear section of the $\mathrm{E}_6$ invariant cubic $\mathcal{C}_{\mathrm{E}_6} \subset \mathbb{P}(V_{27})$, where $V_{27}$ is the minuscule $27$-dimensional representation of $\mathrm{E}_6$. Denote by $L_X \subset V_{27}$ the $9$-dimensional vector space such that $\mathcal{C} \cap \mathbb{P}(L_X) = X$. As explained in \cite{kuz2}, we have a semi-orthogonal decomposition:
 
 \[ \DB(X) = \langle \A_X, \OO_{X}, \OO_{X}(1), \ldots, \OO_{X}(5) \rangle, \]
where $\A_X$ is a $3$-dimensional Calabi-Yau category. We will prove below that the homological unit of $\A_X$ (with respect to the rank function coming from $X$) is $\mathbb{C} \oplus \mathbb{C}[3]$ by exhibiting a $3$-spherical vector bundle in $\A_X$.

Recall that a Jordan algebra structure can be put on $V_{27}$. Namely let :
\begin{equation*}
 V_{27} = 
 \left\{ 
 \begin{pmatrix}
 \lambda_1 & \sigma_1 & \overline{\sigma_2} \\
 \overline{\sigma_1} & \lambda_2 & \sigma_3 \\
 \sigma_2 & \overline{\sigma_3} & \lambda_3\\
 \end{pmatrix},
  \ \sigma_1,\sigma_2,\sigma_3 \in \mathbb{O} \ \textrm{and} \ \lambda_1, \lambda_2, \lambda_3 \in \mathbb{C}
  \right\},
  \end{equation*}
 where $\mathbb{O}$ is the algebra of complexified octonions and $\overline{\sigma}$ is the conjugate of $\sigma$ with respect to the octonionic conjugation. For any $A,B \in V_{27}$, we put $A \star B = \dfrac{AB + BA}{2}$. This is a commutative (but non-associative) product on $V_{27}$, which endows $V_{27}$ with a structure of Jordan algebra.

The determinant of any element of $V_{27}$ is well defined and we get a determinant map, say $\mathrm{Det} \in S^3 V_{27}^*$. The vanishing locus of $\mathrm{Det}$ in $\mathbb{P}(V_{27})$ is the $\mathrm{E}_6$ invariant cubic. Let us write:

\begin{equation*}
\begin{split}
\sigma_1 & = y_1 \textbf{1} + y_2 \textbf{i} + y_3 \textbf{j} + y_4 \textbf{k} + y_5 \textbf{l} + y_6 \textbf{m} + y_7 \textbf{n} + y_8 \textbf{o} \\
\sigma_2 & = y_9 \textbf{1} + y_{10} \textbf{i} + y_{11} \textbf{j} + y_{12} \textbf{k} + y_{13} \textbf{l} + y_{14} \textbf{m} + y_{15} \textbf{n} + y_{16} \textbf{o} \\
\sigma_2 & = y_{17} \textbf{1} + y_{18} \textbf{i} + y_{19} \textbf{j} + y_{20} \textbf{k} + y_{21} \textbf{l} + y_{22} \textbf{m} + y_{23} \textbf{n} + y_{24} \textbf{o} \\
\lambda_1& = y_{25} \\
\lambda_2 & = y_{26}  \\
\lambda_3 &= y_{27},\\
\end{split}
\end{equation*}
where $\textbf{1}, \textbf{i}, \textbf{j}, \textbf{k}, \textbf{l}, \textbf{m}, \textbf{n}, \textbf{o}$ is a standard basis for $\mathbb{O}$ and $y_1, \ldots, y_{27} \in \mathbb{C}$. Then, a \textit{Macaulay2} \cite{M2} computation shows that:

\begin{equation*}
\begin{split}
& \mathrm{Det} \left(  \begin{pmatrix}
 \lambda_1 & \sigma_1 & \overline{\sigma_2} \\
 \overline{\sigma_1} & \lambda_2 & \sigma_3 \\
 \sigma_2 & \overline{\sigma_3} & \lambda_3\\
 \end{pmatrix} \right) =   \\
 & 2y_1y_9y_{17} - 2y_2y_{10}y_{17}-2y_3y_{11}y_{17}-2y_4y_{12}y_{17}-2y_5y_{13}y_{17}-2y_6y_{14}y_{17}\\
                                   &-2y_7y_{15}y_{17}-2y_8y_{16}y_{17}-2y_2y_9y_{18}-2y_1y_{10}y_{18}-2y_4y_{11}y_{18}+2y_3y_{12}y_{18}\\
                                   &-2y_6y_{13}y_{18}+2y_5y_{14}y_{18}+2y_8y_{15}y_{18}-2y_7y_{16}y_{18}-2y_3y_9y_{19}+2y_4y_{10}y_{19}-2y_1y_{11}y_{19}\\
                                   &-2y_2y_{12}y_{19}-2y_7y_{13}y_{19}-2y_8y_{14}y_{19}+2y_5y_{15}y_{19}+2y_6y_{16}y_{19}-2y_4y_9y_{20}-2y_3y_{10}y_{20}\\
                                   &+2y_2y_{11}y_{20}-2y_1y_{12}y_{20}-2y_8y_{13}y_{20}+2y_7y_{14}y_{20}-2y_6y_{15}y_{20}+2y_5y_{16}y_{20}-2y_5y_9y_{21}\\
                                   &+2y_6y_{10}y_{21}+2y_7y_{11}y_{21}+2y_8y_{12}y_{21}-2y_1y_{13}y_{21}-2y_2y_{14}y_{21}-2y_3y_{15}y_{21}-2y_4y_{16}y_{21}\\
                                   &-2y_6y_9y_{22}-2y_5y_{10}y_{22}+2y_8y_{11}y_{22}-2y_7y_{12}y_{22}+2y_2y_{13}y_{22}-2y_1y_{14}y_{22}+2y_4y_{15}y_{22}\\
                                   &-2y_3y_{16}y_{22}-2y_7y_9y_{23}-2y_8y_{10}y_{23}-2y_5y_{11}y_{23}+2y_6y_{12}y_{23}+2y_3y_{13}y_{23}-2y_4y_{14}y_{23}\\
                                   &-2y_1y_{15}y_{23}+2y_2y_{16}y_{23}-2y_8y_9y_{24}+2y_7y_{10}y_{24}-2y_6y_{11}y_{24}-2y_5y_{12}y_{24}\\
                                   &+2y_4y_{13}y_{24}+2y_3y_{14}y_{24}-2y_2y_{15}y_{24}-2y_1y_{16}y_{24}-y_{17}^2y_{25}-y_{18}^2y_{25}-y_{19}^2y_{25}\\
                                   &-y_{20}^2y_{25}-y_{21}^2y_{25}-y_{22}^2y_{25}-y_{23}^2y_{25}-y_{24}^2y_{25}-y_9^2y_{26}-y_{10}^2y_{26}-y_{11}^2y_{26}-y_{12}^2y_{26}\\
                                   &-y_{13}^2y_{26}-y_{14}^2y_{26}-y_{15}^2y_{26}-y_{16}^2y_{26}-y_1^2y_{27}-y_2^2y_{27}-y_3^2y_{27}-y_4^2y_{27}-y_5^2y_{27}\\
                                   &-y_6^2y_{27}-y_7^2y_{27}-y_8^2y_{27}+y_{25}y_{26}y_{27}
 \end{split} 
\end{equation*}

The Hessian matrix of $\mathrm{Det}$ gives a $27*27$ symmetric matrix with linear entries (say $M$) which is part of a matrix factorization of $\mathrm{Det}$. As a consequence, we get an exact sequence:
 
 \[ 0 \longrightarrow V_{27} \otimes \OO_{\mathbb{P}(V_{27})}(-1) \stackrel{M}\longrightarrow V_{27} \otimes \OO_{\mathbb{P}(V_{27})} \longrightarrow i_*(\E) \longrightarrow 0, \]
 
 where $i_*(\E)$ is the push-forward of a rank-$9$ coherent sheaf on $\mathcal{C}_{\mathrm{E}_6}$. The jumping locus of $\E$ is the singular locus of $\mathcal{C}_{\mathrm{E}_6}$, that is the Cayley plane $\mathbb{OP}^2$ (which is a $16$-dimensional projective variety homogeneous under $\mathrm{E}_6$). 
 \bigskip
 
The intersection $X = \mathcal{C} \cap \mathbb{P}(L_X)$ being transverse, we can restrict the above exact sequence to $\mathbb{P}(L_X)$ and we get:

\begin{equation} \label{eq1}
0 \longrightarrow V_{27} \otimes \OO_{\mathbb{P}(L_X)}(-1) \stackrel{M|_{L_X}}\longrightarrow V_{27} \otimes \OO_{\mathbb{P}(L_X)} \longrightarrow i_*(\E_X) \longrightarrow 0,
\end{equation}
 
 where $i_*(\E_X)$ is the push-forward of a rank-$9$ vector bundle on $\mathcal{C} \cap \mathbb{P}(L_X) = X$. It is shown in \cite{maniliev} that $\E_X$ is a $3$-spherical vector bundle on $X$. In \cite{maniliev}, this result is proved via a detailed cohomological study of $\E$ and its twists by line bundles. Their argument ultimately relies on computer-aided (via \textit{Lie} \cite{Lie}) representation theoretic calculations. We propose a different approach which is also based on computer-aided calculations (via \textit{Macaulay2}).

The exact sequence (\ref{eq1}) shows that $\E_X(-1)$ and $\E_X(-2)$ are in $\A_X$, which is a Calabi-Yau category of dimension $3$. By Serre duality, in order to prove that $\E_X$ a $3$-spherical vector bundle, we only have to prove that:

\[ \HHH(\E_X,\E_X) = \mathbb{C} \ \ \textrm{and} \ \ \mathrm{Ex}^{1}(\E_X,\E_X) = 0.\]

We first prove that $\HHH(\E_X,\E_X) = \mathbb{C}.\mathrm{id}$. It is proved in \cite{orlov1}, there is an equivalence of triangulated categories:

\[ \A_X \simeq \mathrm{GrMF}(\mathrm{Det}|_{L_X}), \]

where $\mathrm{GrMF}(\mathrm{Det}|_{L_X})$ is the homotopy category of graded matrix factorizations of $\mathrm{Det}|_{L_X}$ on the vector space $L_X$. Since the matrix $M|_{L_X}$ is part of a matrix factorization of $\mathrm{Det}|_{L_X}$, we can compute $\HHH(\E_X,\E_X)$ in $ \mathrm{GrMF}(\mathrm{Det}|_{L_X})$.

Let $f \in \HHH_{\mathrm{GrMF}(\mathrm{Det}|_{L_X})}(\E_X,\E_X)$. We can represent it as a pair of morphisms $A,B : \mathbb{C}^{27} \longrightarrow \mathbb{C}^{27}$ which makes the following diagram commutes:

\begin{equation*}
\xymatrix{\mathbb{C}^{27} \otimes \OO_{\mathbb{P}(L_X)}(-1) \ar[dd]^{A} \ar[rr]^{M|_{L_X}} & & \mathbb{C}^{27} \otimes \OO_{\mathbb{P}(L_X)} \ar[rr]^{M^{ad}|_{L_X}} \ar[dd]^{B} & & \mathbb{C}^{27} \otimes \OO_{\mathbb{P}(L_X)}(2) \ar[dd]^{A} \\
& & & & \\
\mathbb{C}^{27} \otimes \OO_{\mathbb{P}(L_X)}(-1) \ar[rr]^{M|_{L_X}} & & \mathbb{C}^{27} \otimes \OO_{\mathbb{P}(L_X)} \ar[rr]^{M^{ad}|_{L_X}} & & \mathbb{C}^{27} \otimes \OO_{\mathbb{P}(L_X)}(2)}
\end{equation*}
 
 where $M^{ad}$ is the unique matrix with quadratic entries in the $y_i$ such that $M  M^{ad} = M^{ad}  M = \mathrm{Det}.\mathrm{I}_{27}$. Let $x_1, \ldots, x_9$ be a system of coordinates for $L_X$, the matrix $M_{L_X}$ can easily be obtained by replacing the $y_i$ by their expression from the $x_j$ (and we know that these relations must be \textit{generic} as $L_X$ is generic in $\mathrm{Gr}(9,V_{27})$). Note that he matrices $A$ and $B$ are with constant coefficients. 
 
 As the first square in the above diagram commutes, we get a system of $729$ equations in the $x_i$ and in the coefficients of $A$ and $B$. Since the $x_i$ are independent, we get $6561$ linear equations in the coefficients of $A$ and $B$. We solve this system of equations with \textit{Macaulay2}. There is a one dimensional space of solutions to this system which is $A = B = \lambda.\mathrm{Id}_{27}$, $\lambda \in \mathbb{C}$. This proves that $\HHH(\E_X,\E_X) = \mathbb{C}$. 
 
\bigskip

We will now show that $\mathrm{Ext}^1(E,E) = 0$. Let $f \in \mathrm{Ext}^1_{\mathrm{GrMF}(\mathrm{Det}|_{L_X})}(\E_X,\E_X)$. We can represent it as a pair of morphisms $A: \mathbb{C}^{27} \otimes \OO_X(-1) \longrightarrow \mathbb{C}^{27}\otimes \OO_{X}$ and $B : \mathbb{C}^{27} \otimes \OO_X \longrightarrow \mathbb{C}^{27} \otimes \OO_{X}(2)$ which makes the following diagram commutes:

\begin{equation} \label{eq3}
\xymatrix{\mathbb{C}^{27} \otimes \OO_{\mathbb{P}(L_X)}(-1) \ar[dd]^{A} \ar[rr]^{M|_{L_X}} & & \mathbb{C}^{27} \otimes \OO_{\mathbb{P}(L_X)} \ar[rr]^{M^{ad}|_{L_X}} \ar[dd]^{B} & & \mathbb{C}^{27} \otimes \OO_{\mathbb{P}(L_X)}(2) \ar[dd]^{A} \\
& & & & \\
\mathbb{C}^{27} \otimes \OO_{\mathbb{P}(L_X)} \ar[rr]^{M^{ad}|_{L_X}} & & \mathbb{C}^{27} \otimes \OO_{\mathbb{P}(L_X)}(2) \ar[rr]^{M|_{L_X}} & & \mathbb{C}^{27} \otimes \OO_{\mathbb{P}(L_X)}(3)}
\end{equation}

We want to prove that the pair $(A,B)$ is zero modulo homotopies, which means that there exists $S,T : \mathbb{C}^{27} \longrightarrow \mathbb{C}^{27}$ such that:

\[ A = M|_{L_X} T + S  M|_{L_X} \ \ \textrm{and} \ \  B = T  M^{ad}|_{L_X} + M^{ad}|_{L_X}  S.\]

Since the fist square of the diagram (\ref{eq3}) commutes, we have $M^{ad}|_{L_X}  A = B  M_{L_X}$. Multiplying this relation by $M|_{L_X}$ on the left gives $\mathrm{Det}|_{L_X}.A = M|_{L_X}  B  M|_{L_X}$. As $A$ is with linear coefficient in the $x_i$, this shows that the coefficients of $ M|_{L_X}  B  M|_{L_X}$ are equal to linear forms in the $x$ multiplied by $\mathrm{Det}|_{L_X}$.

If we could show that this implies that the coefficients of $M|_{L_X}  B$ and $B  M|_{L_X}$ are divisible by $P$, then $T = \dfrac{1}{P}\left( B  M|_{L_X} \right)$ and $S = \dfrac{1}{P} \left( M|_{L_X}  B \right)$ would provide the homotopy we are looking for. We haven't been able to prove such a statement. Instead, we use a \textit{Macaulay2} algorithm to prove that $\mathrm{Ext}^1(\E_X,\E_X) = 0$. We first plug-in the equation of $\mathrm{Det}$ into \textit{Macaulay2}. 

\begin{verbatim}
i_1 : kk = ZZ/313

i_2 : V = kk[x_1..x_9,y_1..y_27]

i_3 : DET =
2y_1y_9y_17-2y_2y_10y_17-2y_3y_11y_17-2y_4y_12y_17-2y_5y_13y_17-2y_6y_14y_17
-2y_7y_15y_17-2y_8y_16y_17-2y_2y_9y_18-2y_1y_10y_18-2y_4y_11y_18+2y_3y_12y_18
-2y_6y_13y_18+2y_5y_14y_18+2y_8y_15y_18-2y_7y_16y_18-2y_3y_9y_19+2y_4y_10y_19
-2y_1y_11y_19-2y_2y_12y_19-2y_7y_13y_19-2y_8y_14y_19+2y_5y_15y_19+2y_6y_16y_19
-2y_4y_9y_20-2y_3y_10y_20+2y_2y_11y_20-2y_1y_12y_20-2y_8y_13y_20+2y_7y_14y_20
-2y_6y_15y_20+2y_5y_16y_20-2y_5y_9y_21+2y_6y_10y_21+2y_7y_11y_21+2y_8y_12y_21
-2y_1y_13y_21-2y_2y_14y_21-2y_3y_15y_21-2y_4y_16y_21-2y_6y_9y_22-2y_5y_10y_22
+2y_8y_11y_22-2y_7y_12y_22+2y_2y_13y_22-2y_1y_14y_22+2y_4y_15y_22
-2y_3y_16y_22-2y_7y_9y_23-2y_8y_10y_23-2y_5y_11y_23+2y_6y_12y_23+2y_3y_13y_23
-2y_4y_14y_23-2y_1y_15y_23+2y_2y_16y_23-2y_8y_9y_24+2y_7y_10y_24-2y_6y_11y_24
-2y_5y_12y_24+2y_4y_13y_24+2y_3y_14y_24-2y_2y_15y_24-2y_1y_16y_24-y_17^2y_25
-y_18^2y_25-y_19^2y_25-y_20^2y_25-y_21^2y_25-y_22^2y_25-y_23^2y_25
-y_24^2y_25-y_9^2y_26-y_10^2y_26-y_11^2y_26-y_12^2y_26-y_13^2y_26-y_14^2y_26
-y_15^2y_26-y_16^2y_26-y_1^2y_27-y_2^2y_27-y_3^2y_27-y_4^2y_27-y_5^2y_27
-y_6^2y_27-y_7^2y_27-y_8^2y_27+y_25y_26y_27
\end{verbatim} 

We then create the Hessian matrix of $\mathrm{DET}$:

\begin{verbatim}
i_4 : M = mutableMatrix(V,27,27)

i_5 : for k from 1 to 27 do (for i from 1 to 27 do 
(M_(k-1,i-1) = diff(y_k,diff(y_i,DET))))
\end{verbatim}
We create the generic equations of $L_X$ and we substitute the $y_i$ by these equations to get $\mathrm{Det}|_{L_X}$ and $M|_{L_{X}}$:
\begin{verbatim}
i_6 : K = random(ZZ^27,ZZ^9,Height=>30)

i_7 : K2 = K**V

i_8 : K3 = K2 * matrix{{x_1},{x_2},{x_3},{x_4},{x_5},{x_6},{x_7},{x_8},{x_9}}

i_9 : for j from 1 to 27 do DET = sub(DET, y_j => K3_(j-1,0))

i_10 : for i from 1 to 27 do (for k from 1 to 27 do
(for j from 1 to 27 do M_(k-1,i-1) = sub(M_(k-1,i-1), y_j => K3_(j-1,0))))

i_11 : MX= matrix(M)
\end{verbatim}
We finally compute the self-extensions of the coker of $M|_{L_X}$ over $\mathrm{Proj}\left( \mathbb{C}[x_1, \ldots, x_9]/(\mathrm{Det|_{L_X}}) \right)$.

\begin{verbatim}
i_12 : W = kk[x_1..x_9]

i_13 : DETW= sub(DET,W)

i_14 : MXW = sub(MX,W)

i_15 : J = ideal(DETW)

i_16 : WJ = W/J

i_17 : MXJ = MXW**WJ

i_18 : FX = coker MXJ

i_19 : X = Proj WJ

i_20 : EX = sheaf FX

i_21 : Ext^1(EX,EX)
\end{verbatim}
In about one hour and thirty minutes on a portable workstation, we get the desired answer:

\begin{verbatim}
o_21 : 0
\end{verbatim}
Computations over finite fields are certified by Macaulay$2$. The result above shows that $\mathrm{Ext}^1(\E_X,\E_X) = 0$ over $\mathbb{Z} \slash 313.\mathbb{Z}$. Hence, by semi-continuity, we have $\mathrm{Ext}^1(\E_X,\E_X) = 0$ over $\mathbb{Q}$ and then, the same holds over $\mathbb{C}$. This concludes the proof that $\E_X(-1)$ and $\E_X(-2)$ are $3$-spherical vector bundles in $\A_X$. Proposition \ref{prop1} then implies that the homological unit of $\A_X$ with respect to the rank function coming from $X$ is $\mathbb{C} \oplus \mathbb{C}[3]$.\\

\item Let $X \subset \mathbb{P}(1,1,1,1,1,1,2)$ be a generic double quartic fivefold. We checked in \cite{DQF} that $X$ is a linear section of the double cover of $\mathbb{P}(\Delta_+)$ ramified along the $\mathrm{Spin}_{12}$ invariant quartic $\mathcal{Q} \subset \mathbb{P}(\Delta_+)$ (here $\Delta_{+}$ is one of the $32$-dimensional half-spin representations for $\mathrm{Spin}_{12}$). Denote by $L_X \subset \Delta_+$ the $6$ dimensional vector space such that $X$ is the double cover of $\mathbb{P}(L_X)$ along $\mathcal{Q} \cap \mathbb{P}(L_X)$. As explained in \cite{kuz2}, we have a semi-orthogonal decomposition:
 \[ \DB(X) = \langle \A_X, \OO_{X}, \OO_{X}(1), \OO_{X}(2), \OO_{X}(3) \rangle, \]
where $\A_X$ is a $3$-dimensional Calabi-Yau category. We will prove below that the homological unit of $\A_X$ (with respect to the rank function coming from $X$) is $\mathbb{C} \oplus \mathbb{C}[3]$ by exhibiting a $3$-spherical vector bundle in $\A_X$.

Let us fix a system of coordinates on $\Delta_+^*$. Using the theory of exceptional quaternionic representations, we constructed in \cite{DQF} a $12 \times 12$ matrix, say $M$, with quadratic entries in the variables of $\Delta$ such that $M \times M = P_{\mathcal{Q}} . \mathrm{I}_{12}$, where $P_{\mathcal{Q}}$ is an equation for $\mathcal{Q}$ in the chosen variables of $\Delta_+$. As a consequence, we have an exact sequence:

\[ 0 \longrightarrow \mathbb{C}^{12} \otimes \OO_{\mathbb{P}(\Delta_+)}(-2) \stackrel{M}\longrightarrow \mathbb{C}^{12} \otimes \OO_{\mathbb{P}(\Delta_+)} \longrightarrow i_*\left( \F \right) \longrightarrow 0, \]

where $i_*(\F)$ is the push-forward of a rank-$6$ coherent sheaf on $\mathcal{Q}$. The jumping locus of $\F$ is the singular locus of $\mathcal{Q}$, that is the closure in $\mathbb{P}(\Delta_+)$ of a $24$-dimensional quasi-projective variety homogeneous under $\mathrm{Spin}_{12}$.

\bigskip

The intersection $X = \mathcal{Q} \cap \mathbb{P}(L_X)$ being transverse, we can restrict the above exact sequence to $\mathbb{P}(L_X)$ and we get:

\begin{equation} \label{eq2}
0 \longrightarrow \mathbb{C}^{12} \otimes \OO_{\mathbb{P}(L_X)}(-2) \stackrel{M}\longrightarrow \mathbb{C}^{12} \otimes \OO_{\mathbb{P}(L_X)} \longrightarrow i_*(\F_X) \longrightarrow 0,
\end{equation}
 
 where $i_*(\F_X)$ is the push-forward of a rank-$6$ vector bundle on $\mathcal{Q} \cap \mathbb{P}(L_X) = X$. We showed in \cite{DQF} that $\F_X$ is a $3$-spherical vector bundle on $X$. Furthermore, if we consider the semi-orthogonal decomposition:
 
 \[ \DB(X) = \langle \A_X, \OO_{X}, \OO_{X}(1), \ldots, \OO_{X}(3) \rangle \]
described above, the exact sequence (\ref{eq2}) shows that $\F_X(-1)$ and $\F_X(-2)$ are in $\A_X$. As they are $3$-spherical vector bundles, proposition \ref{prop1} allows to conclude that the homological unit of $\A_X$ with respect to the rank function coming from $X$ is $\mathbb{C} \oplus \mathbb{C}[3]$.

}
\end{enumerate}
\end{exem}

\begin{rem}
Let $X$ be a generic cubic fourfold an $\A_X$ the $K3$-category associated to $X$. It is shown in \cite{huy1} that for $\A_X$ does not contain any $2$-spherical object. We nevertheless prove in proposition \ref{prophodge} that the homological unit of $\A_X$ is $\mathbb{C} \oplus \mathbb{C}[2]$. This hilights the fact that the homological unit is useful in order to capture the <<topological>> properties of a Calabi-Yau category even in the absence of spherical objects.
\end{rem}

There is an analogous statement to that of \ref{prop1} holds for Calabi-Yau categories coming from quivers with potentials, namely we have:

\begin{prop} \label{quiver}
Let $(\mathrm{Q},W)$ be a quiver without loops and $W$ be a reduced potential for $\mathrm{Q}$. Let $\A_{(\mathrm{Q},W)}$ be the Calabi-Yau-$3$ algebra associated to $(\mathrm{Q},W)$. Let $D_{fd}(\A_{(\mathrm{Q},W)})$ be the derived category of finite dimensional DG $\A_{(\mathrm{Q},W)}$-modules. For $\E \in D_{fd}(\A_{(\mathrm{Q},W)})$, we let $\mathrm{rank}(\E)$ be the rank of $\E$ as a representation of $\mathrm{Q}$. Then, the homological unit of $D_{fd}(\A_{(\mathrm{Q},W)})$ with respect to $\mathrm{rank}$ is $\mathbb{C} \oplus \mathbb{C}[3]$.
\end{prop}

We refer to \cite{ginz, keyang} for the construction of $\A_{(Q,W)}$ and to \cite{ginz, keke} for proofs that the category $D_{fd}(\A_{(Q,W)})$ is Calabi-Yau of dimension $3$.

\begin{proof}
We denote by $\mathfrak{T}^{\bullet}_{D_{fd}(\A_{(Q,W)})}$ a homological unit of $\A_{(Q,W)}$ with respect to the $\mathrm{rank}$. Let $i$ be a vertex of $Q$ and let $S_i$ be the simple module corresponding to $i$. Its rank as a $Q$-module is $1$. Since there are no loops at $i$, we know that $\mathrm{Hom}^{\bullet}(S_i, S_i) = \mathbb{C} \oplus \mathbb{C}[3]$ (see lemma 2.15 in \cite{keyang}, for instance). As $\mathrm{rank}(S_i) = 1 \neq 0$, we have an injection of graded algebras:

\[ \mathfrak{T}^{\bullet}_{D_{fd}(\A_{(Q,W)})} \hooklongrightarrow \mathrm{Hom}^{\bullet}(S_i,S_i) = \mathbb{C} \oplus \mathbb{C}[3]. \]

\bigskip

We now prove that $\mathbb{C} \oplus \mathbb{C}[3]$ is indeed the homological unit of $D_{fd}(\A_{(Q,W)})$ with respect to the rank function coming from representations of $Q$. By \cite{politrob}, for any $\E \in D_{fd}(A_{(Q,W)})$, we have a linear map (called the \textit{bulk-boundary map} in \cite{politrob}):

\[ \Theta_{\E} : \mathrm{Hom}^{\bullet}(\E,\E) \longrightarrow \HH_{0}(D_{fd}(\A_{(Q,W)})). \]

The vector $\Theta_{\E}(id_\E) \in \HH_{0}(D_{fd}(\A_{(Q,W)}))$ is called the \textit{Chern character} of $\E$. The set $\{S_i \}_{i \in Q}$ is a minimal system of generators of $D_{fd}(\A_{(Q,W)})$, from which we deduce that the set $\{\Theta_{S_i}(id_{S_i}) \}_{i \in Q}$ is a free family of $\HH_0(\D_{fd}(\A_{(Q,W)}))$ as a $\mathbb{C}$-vector space. Let $\Theta_{S_i}(id_{S_i})^*$ be the linear form on $\HH_{0}(D_{fd}(\A_{(Q,W)}))$ which is orthogonal to $\Theta_{S_i}(id_{S_i})$ (we therefore fix a complementary subspace of $\mathrm{Vect}\{\Theta_{S_i}(id_{S_i}) \}$ in $\HH_{0}(D_{fd}(\A_{(Q,W)}))$). For any $\E \in D_{fd}(\A_{(Q,W)})$, we denote by:

\[ t^0_{\E} := \left( \bigoplus_{i \in Q} \Theta_{S_i}(id_{S_i})^* \right) \circ \Theta_{\E} : \HHH(\E,\E) \longrightarrow \mathbb{C}.\]

It is easily checked that $t_{\E}(id_{\E}) = \mathrm{rank}(\E)$, for any $\E \in D_{fd}(\A_{(Q,W)})$. For all $\E \in  D_{fd}(\A_{(Q,W)})$, we then we define:
\begin{equation*}
\begin{split}
 i^0_{\E}(a) : & \ \mathbb{C} \longrightarrow \HHH(\E,\E) \\
                 & \ a \longmapsto a.id_{\E}
 \end{split}
\end{equation*}

It is obvious that $i_{\E}$ is functorial in $\E$ and that $t^0_{\E}(i^0_{\E})(a) = a.\mathrm{rank}(\E)$, for any $\E \in D_{fd}(\A_{(Q,W)})$. Furthermore, as $D_{fd}(\A_{(Q,W)})$ is Calabi-Yau category of dimension $3$, we have functorial isomorphisms:

\[ \mathrm{Ext}^{3}(\E,\E) \simeq \HHH(\E,\E)^*, \]
for all $\E \in D_{fd}(\A_{(Q,W)})$. As a consequence, the pair of morphisms:

\[ i^{3}_{\E} = \left(t^{0}_{\E} \right)^* : \mathbb{C} \longrightarrow \mathrm{Ext}^{3}(\E,\E), \ \ t^{3}_{\E} = \left(i_{\E}^{0} \right)^{*} : \mathrm{Ext}^{3}(\E,\E) \longrightarrow \mathbb{C} \]
is well-defined and satisfies $t_{\E}^{3} \circ i_{\E}^{3} = \left(t_{\E}^{0} \circ i_{\E}^{0}\right)^{*} = \mathrm{rank}(\E).\mathrm{id}_{\mathbb{C}}$. The algebra structure on $\mathbb{C} \oplus \mathbb{C}[3]$ is trivial. Hence, the functoriality of the isomorphisms $\HHH(\E,\E)^* \simeq \mathrm{Ext}^{3}(\E,\E)$, the functoriality of $i_{\E}$ and the properties of $t_{\E}$ imply that $\mathbb{C} \oplus \mathbb{C}[3]$ is indeed a homological unit for $D_{fd}(\A_{(Q,W)})$.

\end{proof}

\end{subsection}

\begin{subsection}{Invariance of the homological unit}
In this section, we will be interested in the following question:

\begin{quest}
Let $\A$ be a triangulated category and let $\mathrm{rank}_1$, $\mathrm{rank}_2$ be two non-trivial rank functions on $\A$. We denote by $\mathfrak{T}_{\A,1}^{\bullet}$, $\mathfrak{T}_{\A,2}^{\bullet}$ homological units of $\A$ with respect to $\mathrm{rank}_1$ and $\mathrm{rank}_2$. When do we have a ring isomorphism $\mathfrak{T}_{\A,1}^{\bullet} \simeq \mathfrak{T}_{\A,2}^{\bullet}$?
\end{quest}

In the geometric setting, a special case of the above question is the:

\begin{quest}
Let $X,Y$ be smooth projective varieties (over $\mathbb{C}$). Let $\A_X$ and $\A_Y$ be  full admissible subcategories of $\DB(X)$ and $\DB(Y)$. Let $\Phi : \DB(X) \longrightarrow \DB(Y)$ be a Fourier-Mukai functor such that $\Phi$ induces an equivalence between $\A_X$ and $\A_Y$. Let $\mathfrak{T}_{\A_X}^{\bullet}$ and $\mathfrak{T}_{\A_Y}^{\bullet}$ be homological units of $\A_X$ and $\A_Y$ with respect to the rank function coming from $\DB(X)$ and $\DB(Y)$. Assume that $\A_X$ and $\A_Y$ both contain an object which rank is not zero. When do we have a ring isomorphism $\mathfrak{T}_{\A_X}^{\bullet} \simeq \mathfrak{T}_{\A_Y}^{\bullet}$?
\end{quest}

When $\A_X = \DB(X)$ and $\A_Y = \DB(Y)$, we proved in \cite{homounit} that there is a  ring isomorphism $\mathfrak{T}_{\A_X}^{\bullet} \simeq \mathfrak{T}_{\A_Y}^{\bullet}$ provided that \textbf{one} of the following conditions holds:
\begin{itemize}
\item \textit{The kernel giving the equivalence is generically a (possibly shifted) vector bundle on $X \times Y$},
\item \textit{the Chern classes of the kernel giving the equivalence vanish in degree less than $2 \dim X -1$},
\item \textit{the Hodge algebra $\bigoplus_{p=0}^{\dim X} H^{p}(X,\Omega^p_X)$ and $\bigoplus_{p=0}^{\dim X} H^{p}(Y,\Omega^p_Y)$ are generated in degree $1$},
\item \textit{the varieties $X$ and $Y$ have dimension less or equal to $4$}.
\end{itemize}
In case $\A_X = \DB(X)$ and $\A_Y = \DB(Y)$, we expect that there is always a ring isomorphism $\mathfrak{T}_{\A_X}^{\bullet} \simeq \mathfrak{T}_{\A_Y}^{\bullet}$, but we haven't been able to prove it. 

\bigskip

Below, we prove an invariance result for the homological unit of any triangulated category provided that this category contains enough unitary objects with respect to the homological unit under study. We will provide many examples where this result can be applied.

\begin{theo} \label{invaunit}
\begin{enumerate}
\item Let $\A$ be a triangulated category and let $\mathrm{rank}_1$, $\mathrm{rank}_2$ be two non-trivial rank functions on $\A$. We denote by $\mathfrak{T}_{\A,1}^{\bullet}$, $\mathfrak{T}_{\A,2}^{\bullet}$ the homological units associated to $\mathrm{rank}_1$ and $\mathrm{rank}_2$. Let $$\mathrm{cl} : \A \longrightarrow \mathrm{K}_{0}(\A)$$ be the class map and let $\A_{unitary}^{(1)}$ the subset of $\A$ consisting of unitary objects with respect to $\mathfrak{T}_{\A,1}^{\bullet}$. Assume that $\mathrm{cl}(\A_{unitary}^{(1)})$ generates $\mathrm{K}_0(\A) \otimes \mathbb{C}$. Then, there is a injection of graded rings:

\[ \mathfrak{T}_{\A,2}^{\bullet} \hooklongrightarrow \mathfrak{T}_{\A,1}^{\bullet}.\]
\item The same conclusion as above holds if we only assume that $\mathrm{cl}(\A_{unitary}^{(1)})$ generates $\mathrm{K}_{num}(\A) \otimes \mathbb{C}$ provided that $\mathrm{rank}_1$ and $\mathrm{rank}_2$ are numerical rank functions on $\A$.
\end{enumerate}
\end{theo}

\begin{proof}
By definition of rank functions, both $\mathrm{rank}_1$ and $\mathrm{rank}_2$ can be lifted to rank functions:

\begin{equation*}
\mathrm{rank}_1 \otimes \mathbb{C}, \, \mathrm{rank}_2 \otimes \mathbb{C} \ : \mathrm{K}_{0}(\A) \longrightarrow \mathbb{C}.
\end{equation*}
The function $\mathrm{rank}_2$ is non trivial, so the same holds for $\mathrm{rank}_2 \otimes \mathbb{C}$. Furthermore, we know by hypothesis that $\mathrm{cl}(\A_{unitary}^{(1)})$ generates $\mathrm{K}_{0}(\A)$. Hence, there exists $\F \in \A_{unitary}^{(1)}$ such that $\mathrm{rank}_2 \otimes \mathbb{C}(\F) \neq 0$, that is $\mathrm{rank}_2(\F) \neq 0$. By definition of homological units, we have a graded ring embedding:

\[ \mathfrak{T}_{\A,2}^{\bullet} \hooklongrightarrow \HHH^{\bullet}(\F, \F). \]

Since $\F \in \A_{unitary}^{(1)}$, this turns into a graded ring embedding:

\[ \mathfrak{T}_{\A,2}^{\bullet} \hooklongrightarrow \mathfrak{T}_{\A,1}^{\bullet}.\]

\bigskip

\noindent The proof in the numerical case is exactly the same.
\end{proof}

\bigskip

\noindent As a consequence of the above result, we get an effective criterion to determine whether the homological units related to different rank functions on a given triangulated category are isomorphic:

\begin{cor} \label{corrro1}
\begin{enumerate}
\item Let $\A$ be a triangulated category and let $\mathrm{rank}_1$, $\mathrm{rank}_2$ be two non-trivial rank functions on $\A$. We denote by $\mathfrak{T}_{\A,1}^{\bullet}$, $\mathfrak{T}_{\A,2}^{\bullet}$ the homological units associated to $\mathrm{rank}_1$ and $\mathrm{rank}_2$. Assume that both $\mathrm{cl}(\A_{unitary}^{(1)})$ and $\mathrm{cl}(\A_{unitary}^{(2)})$ generate $\mathrm{K}_{0}(\A) \otimes \mathbb{C}$. Then, there is an isomorphism of graded rings:

\[ \mathfrak{T}_{\A,2}^{\bullet} \simeq \mathfrak{T}_{\A,1}^{\bullet}.\]

\item The same conclusion as above holds if we only assume that $\mathrm{cl}(\A_{unitary}^{(1)})$ and $(\A_{unitary}^{(2)})$ generate $\mathrm{K}_{num}(\A) \otimes \mathbb{C}$ provided that $\mathrm{rank}_1$ and $\mathrm{rank}_2$ are numerical rank functions on $\A$.
\end{enumerate}
\end{cor}

Theorem \ref{invaunit} can also be fruitfully applied when the given category is Calabi-Yau of dimension $p \geq 1$ and the homological unit related to a rank function is $\mathbb{C} \oplus \mathbb{C}[p]$. If the classes of unitary objects related to this homological unit generate $\mathrm{K}(\A) \otimes \mathbb{C}$, then the homological unit associated to any other rank function on $\A$ is necessarily $\mathbb{C} \oplus \mathbb{C}[p]$.

\begin{cor} \label{corrro2}
\begin{enumerate}
\item Let $\A$ be a triangulated category which is Calabi-Yau category of dimension $p \geq 1$. Let $\mathrm{rank_1}$ be a non-trivial rank function on $\A$ such that $\mathbb{C} \oplus \mathbb{C}[p]$ is a homological unit for $\A$ with respect to $\mathrm{rank}_1$. Assume that $\mathrm{cl}(\A_{unitary}^{(1)})$ generates $\mathrm{K}_{0}(\A) \otimes \mathbb{C}$. Let $\mathrm{rank}_2$ be another non-trivial rank function on $\A$ and let $\mathfrak{T}_{\A,2}^{\bullet}$ be a homological unit on $\A$ with respect to $\mathrm{rank}_2$. Assume that there exists $\E \in \A$ which is unitary with respect to $\mathfrak{T}_{\A,2}^{\bullet}$. Then we have a graded ring isomorphism:

\[ \mathfrak{T}_{\A,2}^{\bullet} \simeq \mathbb{C} \oplus \mathbb{C}[p].\]
\item The same conclusion as above holds if we only assume that $\mathrm{cl}(\A_{unitary}^{(1)})$ generates $\mathrm{K}_{num}(\A) \otimes \mathbb{C}$ provided that $\mathrm{rank}_1$ and $\mathrm{rank}_2$ are numerical rank functions on $\A$.
\end{enumerate}
\end{cor}

\begin{proof}
By Theorem \ref{invaunit}, we have a graded injection:

\[ \mathfrak{T}_{\A,2}^{\bullet} \hooklongrightarrow \mathbb{C} \oplus \mathbb{C}[p].\] Let $\E \in _A$ be a unitary object with respect to $\mathfrak{T}_{\A,2}^{\bullet}$. Since $\A$ is a Calabi-Yau category of dimension $p$, we have a graded injection (see for instance the proof of Proposition \ref{prop1}):

\[ \mathbb{C} \oplus \mathbb{C}[p] \hooklongrightarrow \HHH^{\bullet}(\E,\E) \simeq \mathfrak{T}_{\A,2}^{\bullet} .\]

We conclude that there is a graded ring isomorphism : $\mathfrak{T}_{\A,2}^{\bullet} \simeq \mathbb{C} \oplus \mathbb{C}[p]$.
\end{proof}

\begin{exem} \label{exemhomounit} There are numerous situations where corollaries \ref{corrro1} and \ref{corrro2} apply:

\begin{enumerate}
\item Let $X$ be a smooth projective threefold over $\mathbb{C}$ with $K_X \simeq \OO_X$. All line bundles on $X$ are unitary objects. Because numerical and homological equivalence coincide on $X$, the complexified cohomological Chern character gives an injection:

\[ \mathrm{Ch}_X(.) \otimes \mathbb{C} : \mathrm{K}_{num}(X) \otimes \mathbb{C} \hooklongrightarrow \bigoplus_{p \geq 0} H^{p}(X, \Omega^p_X).\]

Furthermore, the Lefschetz $1-1$ Theorem and the Hard Lefschetz Theorem easily imply that the Chern characters of line bundles generate $\bigoplus_{p \geq 0} H^{p}(X, \Omega^p_X)$. We deduce that $\mathrm{cl}(\DB(X)_{unitary})$ generates $K_{num}(X) \otimes \mathbb{C}$. In particular, if $X$ is a strict Calabi-Yau variety (that is $H^{\bullet}(\OO_X) = \mathbb{C} \oplus \mathbb{C}[3]$), then corollary \ref{corrro2} shows that the homological unit associated to any other non-trivial numerical rank function on $\DB(X)$ having a unitary object is necessarily $\mathbb{C} \oplus \mathbb{C}[3]$.\\

\item  Let $X \subset \mathbb{P}^{8}$ be a generic cubic hypersurface. Consider the semi-orthogonal decomposition:
 \[ \DB(X) = \langle \A_X, \OO_{X}, \OO_{X}(1), \ldots, \OO_{X}(5) \rangle, \]
where $\A_X$ is a $3$-dimensional Calabi-Yau category. As explained in example \ref{exem1}, there exists a rank $9$ vector bundle $\E_X$ on $X$ which is $3$-spherical and such that $\E_X(-1)$ and $\E_X(-2)$ are in $\A_X$. From this, we deduced that the homological unit of $\A_X$ with respect to the rank function coming from $\DB(X)$ is $\mathbb{C} \oplus \mathbb{C}[3]$ (see proposition \ref{prop1}).
 
One easily computes that $\mathrm{K}_0(\A_X) = \mathbb{Z}^2$ so that the class of $\E_X(-1)$ and $\E_X(-2)$ generate $\mathrm{K}_0(\A_X)$. In particular, $\mathrm{cl}((\A_X)_{unitary})$ generate $\mathrm{K}_0(\A_X) \otimes \mathbb{C}$. Corollary \ref{corrro2} then shows that the homological unit associated to any other non-trivial rank function on $\A_X$ having a unitary object is necessarily $\mathbb{C} \oplus \mathbb{C}[3]$.\\

\item  Let $X \subset \mathbb{P}(1,1,1,1,1,1,2)$ be a generic double quartic fivefold. Consider the semi-orthogonal decomposition:
\[ \DB(X) = \langle \A_X, \OO_{X}, \OO_{X}(1), \OO_{X}(2), \OO_{X}(3) \rangle, \]
where $\A_X$ is a $3$-dimensional Calabi-Yau category. As explained in example \ref{exem1}, there exists a rank $6$ vector bundle $\F_X$ on $X$ which is $3$-spherical and such that $\F_X(-1)$ and $\F_X(-2)$ are in $\A_X$. From this, we deduced that the homological unit of $\A_X$ with respect to the rank function coming from $\DB(X)$ is $\mathbb{C} \oplus \mathbb{C}[3]$ (see proposition \ref{prop1}).
 
One easily computes that $\mathrm{K}_0(\A_X) = \mathbb{Z}^2$ so that the class of $\F_X(-1)$ and $\F_X(-2)$ generate $\mathrm{K}_0(\A_X)$. In particular, $\mathrm{cl}((\A_X)_{unitary})$ generate $\mathrm{K}_0(\A_X) \otimes \mathbb{C}$. Corollary \ref{corrro2} then shows that the homological unit associated to any other non-trivial rank function on $\A_X$ having a unitary object is necessarily $\mathbb{C} \oplus \mathbb{C}[3]$.\\

\item Let $(\mathrm{Q},W)$ be a quiver without loops and $W$ be a reduced potential for $\mathrm{Q}$. Let $\A_{(\mathrm{Q},W)}$ be the Calabi-Yau-$3$ algebra associated to $(\mathrm{Q},W)$. Let $D_{fd}(\A_{(\mathrm{Q},W)})$ be the derived category of finite dimensional DG $\A_{(\mathrm{Q},W)}$-modules. For $\E \in D_{fd}(\A_{(\mathrm{Q},W)})$, we let $\mathrm{rank}(\E)$ be the rank of $\E$ as a representation of $\mathrm{Q}$. As proved in proposition \ref{quiver}, the homological unit of $D_{fd}(\A_{(\mathrm{Q},W)})$ with respect to $\mathrm{rank}$ is $\mathbb{C} \oplus \mathbb{C}[3]$. 

Let $i$ be a vertex of $Q$ and let $S_i$ be the simple module corresponding to $i$. We recalled in the proof of proposition \ref{quiver} that the $\{S_i \}_{i \in Q}$ are $3$-spherical objects which generate $D_{fd}(\A_{(\mathrm{Q},W)})$. Hence, $\mathrm{cl}\left((D_{fd}(\A_{(\mathrm{Q},W)})_{unitary} \right)$ generate $\mathrm{K}_0(\A_{(\mathrm{Q},W)}) \otimes \mathbb{C}$. Corollary \ref{corrro2} then shows that the homological unit associated to any other non-trivial rank function on $D_{fd}(\A_{(\mathrm{Q},W)})$ having a unitary object is necessarily $\mathbb{C} \oplus \mathbb{C}[3]$.\\
\end{enumerate}

\end{exem}
\end{subsection}

\end{section}

\begin{section}{A Hodge structure on the Hochschild homology of some Calabi-Yau categories of dimension $3$}
\begin{subsection}{Hodge numbers for Calabi-Yau categories of dimension $3$}

In this section, using the homological unit, we define Hodge numbers for Calabi-Yau categories of dimension $3$ and we provide examples of computations of such numbers.

\begin{defi} \label{defihodgenumbers}
Let $\A$ be a smooth proper triangulated category which is Calabi-Yau of dimension $3$. Let $\mathrm{rank}$ be a non-trivial rank function on $\A$ and let $\mathfrak{T}_{\A}^{\bullet}$ be a homological unit for $\A$ with respect to $\mathrm{rank}$. We define the \textit{Hodge numbers} of $\A$ as:
\begin{enumerate}
\item for all $i \in [0, \ldots, 3]$, $h^{i,0}(\A) = \mathfrak{T}_{\A}^{3-i}$,
\item $h^{3,1}(\A) = \dim \HH_{-2}(\A) - h^{2,0}(\A)$,
\item $h^{3,2}(\A) = h^{1,0}(\A)$ and $h^{2,1}(\A) = \dim \HH_{-1}(\A) - h^{1,0}(\A) - h^{3,2}(\A)$,
\item $h^{3,3}(\A) = h^{0,0}(\A)$ and $h^{1,1}(\A) = h^{2,2}(\A) = \dfrac{\dim \HH_0(\A) - h^{0,0}(\A) - h^{3,3}(\A)}{2}$.
\item $h^{p,q}(\A) = h^{q,p}(\A)$ for any $p,q \in [0, \ldots, 3]$.
\end{enumerate}

\end{defi}

In case $\A = \DB(X)$, where $X$ is a smooth complex projective variety of dimension $3$ with $K_X = \OO_X$, one immediately checks (with the Hochschild-Kostant-Rosenberg isomorphism) that the Hodge numbers defined above match with the <<classical>> Hodge numbers. In the general case, it is not immediately clear that all these Hodge numbers are positive (and integral as far as $h^{1,1}$ and $h^{2,2}$ are concerned). We provide an easy criterion to check that all these numbers (except possibly $h^{2,1}$) are non-negative integers. 

\begin{prop} \label{hodgenumbers}
Let $\A$ be a a smooth proper triangulated category Calabi-Yau of dimension $3$, which is the derived category of perfect $DG$-modules over a $DG$-algebra $\C$. Let $\mathrm{rank}$ be a non-trivial rank function on $\A$ and let $\mathfrak{T}_{\A}^{\bullet}$ be a homological unit for $\A$ with respect to $\mathrm{rank}$. Assume that $\A$ is connected (that is $\HH_{-3}(\A) = \mathbb{C}$) and that there exists a unitary object with respect to $\mathfrak{T}_{\A}^{\bullet}$ in $\A$. Then all $h^{p,q}(\A)$ (except possibly $h^{2,1}(\A) = h^{1,2}(\A)$) are non-negative integers.
\end{prop}

\begin{proof}
We first notice that for all $i \in [0, \ldots, 3]$, the numbers $h^{i,0}(\A)$ are non-negative integers by definition. We also remark that $h^{3,0}(\A) = h^{0,0}(\A) \neq 0$ and that $h^{1,0}(\A) = h^{2,0}(\A)$. Indeed, we know by hypothesis that there exists an unitary object in $\A$ and Serre duality applies to its endomorphism algebra (which contains the identity in degree $0$).

\bigskip

The definition of homological unit implies that there is a graded ring embedding:
\[ \mathfrak{T}_{\A}^{\bullet} \hooklongrightarrow \HH^{\bullet}(\A).\]

Since $\A$ is a Calabi-Yau category of dimension $3$, there is an isomorphism of graded vector spaces $\HH^{\bullet}(\A) \simeq \HH_{\bullet -3}(\A)$. We deduce that there is an embedding of graded vector spaces:

\[ \mathfrak{T}_{\A}^{\bullet} \hooklongrightarrow \HH_{\bullet - 3}(\A).\]
As a consequence, the number $h^{3,1}(\A)$ is a non-negative integer. As noticed earlier,we have $\mathfrak{T}_{\A}^{3} \simeq \left( \mathfrak{T}_{\A}^{0}\right)^* \neq 0$. Moreover, we have $\HH_{-3}(\A) = \mathbb{C}$ by hypothesis (connectivity of $\A$). We find that $\mathbb{C} \simeq \HH_{-3}(\A) \simeq \mathfrak{T}_{\A}^{3}(\A) \simeq \mathfrak{T}_{\A}^{0}$. In particular $h^{3,0}(\A) = h^{0,0}(\A) = 1$.

\bigskip

As mentioned in the statement of proposition \ref{hodgenumbers}, we do not know if $h^{2,1}(\A)$ is always a non-negative integer. We move on to prove that $h^{1,1}(\A) = h^{2,2}(\A)$ are non-negative integers. Since $\A$ is the derived category of perfect $DG$-modules over the $DG$-algebra $\C$, we have an identification:

\[ \HH^{\bullet}(\A)  = \HHH^{\bullet}_{D^{perf}(\C^{op} \otimes \C)}(\Delta_{\C}, \Delta_{\C}),\]
where $\Delta_{\C}$ is the diagonal bimodule over $\C^{op} \otimes \C$. As $\A$ is a Calabi-Yau category of dimension $3$, the category $D^{perf}(\C^{op} \otimes \C)$ is a Calabi-Yau category of dimension $6$. Serre duality then provides a graded-commutative perfect pairing (given by composition of morphisms followed by a trace map):

\[ S_{(\C^{op} \otimes \C)}^{\bullet} :  \HHH^{\bullet}_{D^{perf}(\C^{op} \otimes \C)}(\Delta_{\C}, \Delta_{\C}) \times  \HHH^{6-\bullet}_{D^{perf}(\C^{op} \otimes \C)}(\Delta_{\C}, \Delta_{\C}) \longrightarrow \mathbb{C}.\]
Specializing to the case $\bullet = 3$, we find that $\HHH^{3}_{D^{perf}(\C^{op} \otimes \C)}(\Delta_{\C}, \Delta_{\C}) = \HH^{3}(\A)$ is a symplectic vector space, in particular its dimension is even. We deduce that $\dim \HH_{0}(\A)$ is even. As $\mathbb{C} \simeq \mathfrak{T}_{\A}^{3}(\A)$ embeds in $\HH_{0}(\A)$, the dimension of $\HH_{0}(\A)$ must necessarily be an even integer strictly positive. Since $h^{0,0}(\A) = h^{3,3}(\A) = 1$, we have proved that the numbers $h^{1,1}(\A)$ and $h^{2,2}(\A)$ defined by: 
\[h^{1,1}(\A) = h^{2,2}(\A) =\dfrac{\dim \HH_0(\A) - h^{0,0}(\A) - h^{3,3}(\A)}{2}\] are integral and non-negative.
\end{proof}

Using the results of the previous sections, one can prove the invariance (with respect to the rank function) of the afore-mentioned Hodge numbers in many situations. Namely we have the:

\begin{theo} \label{invahodgenumbers}
\begin{enumerate}
\item Let $\A$ be a smooth proper triangulated category which is Calabi-Yau of dimension $3$. Let $\mathrm{rank}_1, \mathrm{rank}_2$ be a non-trivial rank functions on $\A$ and let $\mathfrak{T}_{\A,1}^{\bullet}, \mathfrak{T}_{\A,2}^{\bullet}$ be homological units for $\A$ with respect to $\mathrm{rank}_1$, $\mathrm{rank}_2$. Let $\mathrm{cl} : \A \longrightarrow \mathrm{K}_{0}(\A)$ be the class map and denote by $\A_{unitary}^{(i)}$ the set of objects in $\A$ which are unitary with respect to $\mathfrak{T}_{\A,i}^{\bullet}$. Finally, for all $p,q \in [0,\ldots,3]$, we denote by $h^{p,q}_i(\A)$ the Hodge numbers of $\A$ associated to $\mathfrak{T}_{\A,i}^{\bullet}$ as in definition \ref{defihodgenumbers}. We have the following:

\begin{enumerate}
\item  If both $\mathrm{cl}(\A_{unitary}^{(1)})$ and $\mathrm{cl}(\A_{unitary}^{(2)})$ generate $\mathrm{K}_{0}(\A) \otimes \mathbb{C}$, then we have:

\[ h^{p,q}_1(\A) = h^{p,q}_2(\A),\]
for all $p,q \in [0,\ldots,3]$.

\item If $\mathfrak{T}_{\A,1}^{\bullet} = \mathbb{C} \oplus \mathbb{C}[3]$, $\mathrm{cl}(\A_{unitary}^{(1)})$ generates $\mathrm{K}_{0}(\A) \otimes \mathbb{C}$ and there exists a unitary object in $\A$ with respect to $\mathfrak{T}_{\A,2}^{\bullet}$, then we have:

\[ h^{p,q}_1(\A) = h^{p,q}_2(\A),\]
for all $p,q \in [0,\ldots,3]$.
\end{enumerate}

\item The same conclusions as above hold if we only assume that $\mathrm{cl}(\A_{unitary}^{(1)})$ and $\mathrm{cl}(\A_{unitary}^{(2)})$ generate $\mathrm{K}_{num}(\A) \otimes \mathbb{C}$ (resp.  $\mathrm{cl}(\A_{unitary}^{(1)})$ generates $\mathrm{K}_{num}(\A) \otimes \mathbb{C}$) provided that $\mathrm{rank}_1$ and $\mathrm{rank}_2$ are numerical rank functions on $\A$.
\end{enumerate}
\end{theo}

\begin{proof}
In light of the definition of Hodge numbers given in \ref{defihodgenumbers}, the statements above are easy consequences of corollaries \ref{corrro1} and \ref{corrro2}.
\end{proof}

\begin{exem} \label{exemhodgenumbers}
\begin{enumerate}
\item  Let $X \subset \mathbb{P}^{8}$ be a generic cubic hypersurface. Consider the semi-orthogonal decomposition:
 \[ \DB(X) = \langle \A_X, \OO_{X}, \OO_{X}(1), \ldots, \OO_{X}(5) \rangle, \]
where $\A_X$ is a $3$-dimensional Calabi-Yau category. As explained in example \ref{exemhomounit}, the homological unit of $\A_X$ with respect to the rank function coming from $\DB(X)$ is $\mathbb{C} \oplus \mathbb{C}[3]$. The Hochschild homology numbers for $X$ are (see \cite{maniliev2}, section 3):

\begin{itemize}
\item $\mathrm{hh}_0(X) = 8$,
\item $\mathrm{hh}_1(X) = \mathrm{hh}_{-1}(X) = 84$,
\item $\mathrm{hh}_2(X) = \mathrm{hh}_{-2}(X) = 1$
\end{itemize}
 
The direct sum decomposition $\HH_{\bullet}(X) = \HH_{\bullet}(\A_X) \oplus \mathbb{C}^6$ finally implies that the Hodge diamond of $\A_X$ with respect to the rank function coming from $\DB(X)$ is:

\begin{center}
\begin{tabular}{ccccccc} 
& & & 1 & & & \\
& & 0& &0 & &  \\
& 0& &0 & &0 & \\
1& &84 & &84 & &1 \\
& 0& &0 & &0 & \\
& & 0& &0 & &  \\
& & & 1 & & & \\
\end{tabular}
\end{center}
As mentioned in example \ref{exemhomounit}, we know that $\mathrm{cl}((\A_X)_{unitary})$ generate $\mathrm{K}_0(\A_X) \otimes \mathbb{C}$. As a consequence, Theorem \ref{invahodgenumbers} guarantees that the Hodge numbers of $\A_X$ defined for any other non-trivial rank function on $\A_X$ having a unitary object are equal to the numbers appearing in the above diamond.

\item Let $X \subset \mathbb{P}(1,1,1,1,1,1,2)$ be a generic double quartic fivefold. Consider the semi-orthogonal decomposition:
\[ \DB(X) = \langle \A_X, \OO_{X}, \OO_{X}(1), \OO_{X}(2), \OO_{X}(3) \rangle, \]
where $\A_X$ is a $3$-dimensional Calabi-Yau category. As explained in example \ref{exemhomounit}, the homological unit of $\A_X$ with respect to the rank function coming from $\DB(X)$ is $\mathbb{C} \oplus \mathbb{C}[3]$. The Hochschild homology numbers for $X$ are (see \cite{maniliev2}, section 3):

\begin{itemize}
\item $\mathrm{hh}_0(X) = 6$,
\item $\mathrm{hh}_1(X) = \mathrm{hh}_{-1}(X) = 90$,
\item $\mathrm{hh}_2(X) = \mathrm{hh}_{-2}(X) = 1$
\end{itemize}
 
The direct sum decomposition $\HH_{\bullet}(X) = \HH_{\bullet}(\A_X) \oplus \mathbb{C}^6$ finally implies that the Hodge diamond of $\A_X$ with respect to the rank function coming from $\DB(X)$ is:

\begin{center}
\begin{tabular}{ccccccc} 
& & & 1 & & & \\
& & 0& &0 & &  \\
& 0& &0 & &0 & \\
1& &90 & &90 & &1 \\
& 0& &0 & &0 & \\
& & 0& &0 & &  \\
& & & 1 & & & \\
\end{tabular}
\end{center}
As mentioned in example \ref{exemhomounit}, we know that $\mathrm{cl}((\A_X)_{unitary})$ generate $\mathrm{K}_0(\A_X) \otimes \mathbb{C}$. As a consequence, Theorem \ref{invahodgenumbers} guarantees that the Hodge numbers of $\A_X$ defined for any other non-trivial rank function on $\A_X$ having a unitary object are equal to the numbers appearing in the above diamond.
\end{enumerate}
\end{exem}

The cubic sevenfold and the double quartic fivefolds are examples of Fano manifolds of Calabi-Yau type that were introduced in \cite{maniliev2}. As far as complete intersections in weighted projective spaces are concerned, there is a another example of Fano manifolds of Calabi-Yau type exhibited in \cite{maniliev2} : the (transverse) complete intersection of a smooth cubic and a smooth quadric in $\mathbb{P}^7$. Let $X$ such a complete intersection. It is known (see \cite{kuz2}) that there is a semi-orthogonal decomposition:

\[ \DB(X) = \langle \A_X, \OO_X, \OO_X(1), \OO_X(2), \mathcal{S}(2) \rangle, \]
where $\A_X$ is Calabi-Yau of dimension $3$ and $\mathcal{S}$ is the restriction of one of the Spinor bundles from $\mathbb{Q}^6$. It is easily computed that $\mathrm{K}_0(\A_X) = \mathbb{Z}^2$.

\begin{quest}
 Can we find a spherical bundle $\E$ on $X$ such that $\E(-1) \in \A_X$ and $\E(-2) \in \A_X$?
\end{quest}
A positive answer to this question would show that the homological unit of $\A_X$ with respect to the rank function coming from $\DB(X)$ is $\mathbb{C} \oplus \mathbb{C}[3]$. Hence, the Hodge diamond of $\A_X$ (with respect to the rank function on $\A_X$ coming from $\DB(X)$) would be:

\begin{center}
\begin{tabular}{ccccccc} 
& & & 1 & & & \\
& & 0& &0 & &  \\
& 0& &0 & &0 & \\
1& &83 & &83 & &1 \\
& 0& &0 & &0 & \\
& & 0& &0 & &  \\
& & & 1 & & & \\
\end{tabular}
\end{center}
Theorem \ref{invahodgenumbers} would guarantee that Hodge numbers of $\A_X$ with respect to any other rank function on $\A_X$ having a unitary object are equal to the above numbers.

\bigskip
\begin{rem}
It would also certainly be desirable to compute the Hodge numbers of the corresponding categories in the case of quiver with potentials. We already know that their homological units are $\mathbb{C} \oplus \mathbb{C}[3]$ and that $\mathrm{cl}(\A_{unitary})$ generates $\mathrm{K}_{0}(\A) \otimes \mathbb{C}$ (see proposition \ref{quiver}). As a consequence of Theorem \ref{invahodgenumbers}, the Hodge numbers of the $3$-Calabi-Yau categories coming from quiver with potentials are independent of the rank function.
\end{rem}

\end{subsection}

\begin{subsection}{A Hodge structure}
In this section we assume that $\A$ is a geometric Calabi-Yau of dimension $3$ : there exists is a smooth projective variety over $\mathbb{C}$, say $X$, and a semi-orthogonal decomposition:
\begin{equation*}
\DB(X) = \langle \A, \leftexp{\perp}{\A} \rangle
\end{equation*}
such that $\A$ is Calabi-Yau of dimension $3$. By the Hochschild-Kostant-Rosenberg isomorphism, we have an isomorphism:
\begin{equation*}
\begin{split}
\tau_{\mathrm{HH}_{\bullet}} : &\ \ \mathrm{HH}_{\bullet}(\DB(X)) \simeq \bigoplus_{p-q = \bullet}H^{q}(X, \Omega^p_X), \\
\end{split}
\end{equation*}
Furthermore, by Hodge symmetry, the complex conjugation induces an isomorphism:
\begin{equation*}
\mathrm{HS}_{X} : H^p(X,\Omega^q_X) \simeq H^{q}(X, \Omega^p_X).
\end{equation*}
Composing complex conjugation with the inverse of the map $\tau_{\mathrm{HH}_{\bullet}}$, we find an involution:

\begin{equation*}
\mathrm{c}_X : \mathrm{HH}_{\bullet}(\DB(X)) \simeq \mathrm{HH}_{-\bullet}(\DB(X))
\end{equation*}
From now, we make the following hypothesis:

\begin{hypo} \label{hypohodge}
The map $\mathrm{c}_X$ stabilizes the Hochschild homology of $\leftexp{\perp}{\A}$ in the decomposition:
\begin{equation*}
\HH_{\bullet}(\DB(X)) = \HH_{\bullet}(\A) \oplus \HH_{\bullet}(\leftexp{\perp}{\A})
\end{equation*}
where $\leftexp{\perp}{\A}$ is the left-orthogonal to $\A$ in $\DB(X)$.
\end{hypo}

This hypothesis is satisfied in many situations which we shall describe. Let $Y_1, \ldots, Y_k$ be smooth projective varieties and $\F_1, \ldots, \F_k$ be objects in $\DB(Y_1 \times X), \ldots, \DB(Y_k \times X)$. We denote by $p_k$ and $q_k$ the natural projections in the diagram:

\begin{equation*}
\xymatrix{  & Y_k \ar[ldd]_{q_k} \times X \ar[rdd]^{p_k}&  \\
& &  \\
Y_k & & X}
\end{equation*}
Assume that the functors:
\begin{equation*}
\Phi_k(?) = (p_k)_*(q_k^*(?) \otimes \F_k) : \DB(Y_k) \longrightarrow \DB(X)
\end{equation*}
are fully faithful and that there is a semi-orthogonal decomposition:

\begin{equation*}
\DB(X) = \langle \A, \Phi_1(\DB(Y_1)), \ldots, \Phi_k(\DB(Y_k)) \rangle.
\end{equation*}
Then hypothesis \ref{hypohodge} holds in that case. In particular, since we have an equality:
\begin{equation*}
\HH_{\bullet}(\A) = \HH_{\bullet}(X)/\HH_{\bullet}(\leftexp{\perp}{\A}),
\end{equation*}
the map $\mathrm{c}_X$ descends to an involution:

\begin{equation*}
\cc_{\A} : \HH_{\bullet}(\A) \simeq \HH_{- \bullet}(\A).
\end{equation*}
Since $\A$ is a semi-orthogonal component of $\DB(X)$ with $X$ smooth projective, we can write $\A$ as the derived category of $DG$-modules over some $DG$-algebra, say $\C$. In particular, we have an identification:

\[ \HH^{\bullet}(\A)  = \HHH^{\bullet}_{D^{perf}(\C^{op} \otimes \C)}(\Delta_{\C}, \Delta_{\C}),\]
where $\Delta_{\C}$ is the diagonal bimodule over $\C^{op} \otimes \C$. Assume that $\A$ is a Calabi-Yau category of dimension $3$, the category $D^{perf}(\C^{op} \otimes \C)$ is then a Calabi-Yau category of dimension $6$. Serre duality then provides a graded-commutative perfect pairing (given by composition of morphisms followed by a trace map):

\[ S_{(\C^{op} \otimes \C)}^{\bullet} :  \HHH^{\bullet}_{D^{perf}(\C^{op} \otimes \C)}(\Delta_{\C}, \Delta_{\C}) \times  \HHH^{6-\bullet}_{D^{perf}(\C^{op} \otimes \C)}(\Delta_{\C}, \Delta_{\C}) \longrightarrow \mathbb{C}.\]
Specializing to the case $\bullet = 3$, we find that $\HHH^{3}_{D^{perf}(\C^{op} \otimes \C)}(\Delta_{\C}, \Delta_{\C}) = \HH^{3}(\A)$ is a symplectic vector space with symplectic form $S_{(\C^{op} \otimes \C)}^{3}$. As $\A$ is Calabi-Yau of dimension $3$, we have an isomorphism $\HH_{0}(\A) \simeq \HH^{3}(\A)$. Hence we can lift $S_{(\C^{op} \otimes \C)}^{3}$ to a symplectic form on $\HH_{0}(\A)$, which we denote by $\omega_{\HH_0(\A)}$.

\begin{defi} \label{defihodgespaces}
Let $X$ be smooth projective variety and let $\A$ semi-orthogonal component of $\DB(X)$ which is Calabi-Yau of dimension $3$ and connected (that is $\HH_{-3}(\A) = \mathbb{C}$). Let $\mathrm{rank}$ be a non-trivial rank function on $\A$ and assume that a homological unit for $\A$ with respect to $\mathrm{rank}$ is $\mathbb{C} \oplus \mathbb{C}[3]$. Assume furthermore that hypothesis \ref{hypohodge} is satisfied. We define the \textit{Hodge spaces} of $\A$ as:
\begin{enumerate}
\item $H^{3,0}(\A) = \HH_{-3}(\A) = \mathbb{C}$, $H^{0,0}(\A) = \mathbb{C} \subset \HH_{0}(\A)$, $H^{1,0}(\A) = H^{2,0}(\A) = 0$.
\item $H^{3,1}(\A) = \HH_{-2}(\A)$, $H^{3,2}(\A) = 0$ and $H^{2,1}(\A) = \HH_{-1}(\A)$.
\item We choose $V_1$ a maximal isotropic subspace of $\HH_{0}(\A)$ (for $\omega_{\HH_{0}(\A)}$) containing $H^{0,0}(\A)$ and $V_2$ a maximal isotropic subspace in $\HH_{0}(\A)$ which is complementary to $V_1$. We let $H^{1,1}(\A)$ be a complementary subspace of $H^{0,0}(\A)$ in $V_1$. We let $H^{3,3}(\A)$ be a line in $V_2$ and $H^{2,2}(\A)$ be a complementary subspace of $H^{3,3}(\A)$ in $V_2$.
\item $H^{p,q}(\A) = \cc_{\A}(H^{q,p}(\A))$ for any $(p,q) \in [0, \ldots, 3]$ such that $p<q$.
\end{enumerate}

\end{defi}

\begin{rem}
\begin{enumerate}
\item By definition of homological units, we have a graded embedding $\mathbb{C} \oplus \mathbb{C}[3] \hooklongrightarrow \HH^{\bullet}(\A)$. The category $\A$ is Calabi-Yau of dimension $3$, so that there is a graded embedding:
\begin{equation*}
\mathbb{C} \oplus \mathbb{C}[3] \hooklongrightarrow \HH_{\bullet-3}(\A).
\end{equation*}
This accounts for the definition of $H^{0,0}(\A)$ and its embedding in $\HH_{0}(\A)$.

\item Since $H^{0,0}(\A)$ is a line in the symplectic vector space $(\HH_{0}(\A), \omega_{\HH_{0}(\A)})$, it is automatically an isotropic subspace of $\HH_{0}(\A)$. The definition of $V_1$ and $H^{1,1}(\A)$ is accordingly meaningful.

\item The definition of $H^{3,3}(\A)$ and $H^{2,2}(\A)$ looks rather arbitrary, but it seems difficult to do better in the absence of the a reasonable Lefschetz operator on $\HH_{0}(\A)$ (which we are not able to provide). Note that for $\A$ the Calabi-Yau $3$ category inside the derived category of the cubic fourfold or the double quartic fivefold, this potential Leschetz operator is probably to be defined as $0$. Indeed, we know $h^{1,1}(\A) = 0$ in both cases.

\item It doesn't seem impossible to extend the definition of these Hodge spaces beyond the case where the homological unit with respect to the chosen rank function is $\mathbb{C} \oplus \mathbb{C}[3]$. One would naturally define $H^{i,0}(\A) = \mathfrak{T}_{\A,3-i}$, for $i \in [0, \ldots, 3]$. The definition of $H^{1,1}(\A)$, $H^{2,2}(\A)$, $H^{3,3}(\A)$ is carried out exactly as in definition \ref{defihodgespaces}. On the other hand, the definition of $H^{3,1}(\A)$, $H^{3,2}(\A)$ and $H^{2,1}(\A)$ looks less obvious, but could probably be found out in many special cases.
\end{enumerate}
\end{rem}

\begin{prop}
Let $X$ be smooth projective variety and let $\A$ semi-orthogonal component of $\DB(X)$ which is Calabi-Yau of dimension $3$ and connected (that is $\HH_{-3}(\A) = \mathbb{C}$). Let $\mathrm{rank}$ be a non-trivial rank function on $\A$ and assume that a homological unit for $\A$ with respect to $\mathrm{rank}$ is $\mathbb{C} \oplus \mathbb{C}[3]$. Assume furthermore that hypothesis \ref{hypohodge} is satisfied. Consider the Hodge spaces of $\A$ as in definition \ref{defihodgespaces}. Then we have a graded decomposition:
\begin{equation*}
\HH_{\bullet}(\A) = \bigoplus_{p-q = \bullet} H^{p,q}(\A)
\end{equation*}
and the direct sum $\ds \bigoplus_{p,q \geq 0} H^{p,q}(\A)$ is a Hodge structure on $\HH_{\bullet}(\A)$. If $\mathrm{cl}(\A_{unitary})$ generates $\mathrm{K}_{0}(\A) \otimes \mathbb{C}$ (resp.  $\mathrm{K}_{num}(\A) \otimes \mathbb{C}$ provided that $\mathrm{rank}$ is a numerical rank function) , then the dimensions of the Hodge spaces defined for any other non-trivial rank function (resp. non-trivial numerical rank function) having a unitary object on $\A$ are equal to the $h^{p,q}(\A) = \dim H^{p,q}(\A)$.
\end{prop}
\begin{proof}
The graded decomposition and the fact that $\ds \bigoplus_{p,q \geq 0} H^{p,q}(\A)$ is a Hodge structure on $\HH_{\bullet}(\A)$ follow immediately from definition \ref{defihodgespaces} and the existence of the involution $c_{\A} : \HH_{\bullet}(\A) \simeq \HH_{- \bullet}(\A)$ which sends $H^{p,q}(\A)$ on $H^{q,p}(\A)$ (by definition). The second part of the proposition follows from Theorem \ref{invahodgenumbers}.
\end{proof}

\end{subsection}
\begin{subsection}{Toward Homological Mirror Symmetry for the cubic sevenfold and the double quartic fivefold}

\begin{subsubsection}{Cubic sevenfold}

Let $T = E \times E \times E$ be the triple product of an elliptic curve $E$, given by the equation $\{Z_1^3+Z_2^3 + Z_3^3 = 0 \} \subset \mathbb{P}^2 $ and let $\mathbb{Z}_3 \times \mathbb{Z}_3$ acts on $T$ as follows:

\begin{equation*}
\begin{split}
(1,0) . (z_1,z_2,z_3,z_4,z_5,z_6,z_7,z_8,z_9) & = (\alpha.z_1, z_2,z_3,\alpha^2 .z_4,z_5,z_6,z_7,z_8,z_9) \\
(0,1) . (z_1,z_2,z_3,z_4,z_5,z_6,z_7,z_8,z_9) & = (\alpha.z_1,z_2,z_3, z_4,z_5,z_6, \alpha^2.z_7,z_8,z_9),
\end{split}
\end{equation*} 
where $\alpha$ is a cubic root of unity. The quotient $T/\mathbb{Z}_3 \times \mathbb{Z}_3$ has crepant resolution which is a Calabi-Yau threefold. We denote it by $Z_1$. The Hodge diamond of $Z_1$ is:

\begin{center}
\begin{tabular}{ccccccc} 
& & & 1 & & & \\
& & 0& &0 & &  \\
& 0& &84 & &0 & \\
1& &0 & &0 & &1 \\
& 0& &84 & &0 & \\
& & 0& &0 & &  \\
& & & 1 & & & \\
\end{tabular}
\end{center}

As explained in \cite{CDP, sch, BBVW}, the mirror of $Z_1$ ought to be a Landau-Ginzburg model related to a smooth cubic sevenfold. Let $X_1 \subset \mathbb{P}^8$ be a smooth cubic sevenfold. We have a semi-orthogonal decomposition:

\begin{equation*}
\DB(X_1) = \langle \A_1, \OO_{X_1}, \ldots, \OO_{X_1}(5) \rangle,
\end{equation*} 
where $\A_{1}$ is a Calabi-Yau category of dimension $3$. It follows from \cite{orlov1} that $\A_1$ is the homotopy category of the $DG$-category of graded matrix factorizations of the equation of $X_1$. Therefore, the category $\A_1$ can be interpreted as a Landau-Ginzburg model for the cubic sevenfold $X_1$. We found out in example \ref{exemhodgenumbers} that the Hodge diamond of $\A_1$ is:

\begin{center}
\begin{tabular}{ccccccc} 
& & & 1 & & & \\
& & 0& &0 & &  \\
& 0& &0 & &0 & \\
1& &84 & &84 & &1 \\
& 0& &0 & &0 & \\
& & 0& &0 & &  \\
& & & 1 & & & \\
\end{tabular}
\end{center}
One observes that both diamonds are obtained from each other by a $\pi/2$-rotation. This is certainly a favorable presage as far as mirror symmetry is concerned. 

\bigskip

We have shown in example \ref{exem1} that there exists a $3$-spherical bundle $\E_{X_1}$ on $X_1$ such that $\E_{X_1}(-1)$ and $\E_{X_1}(-2)$ are in $\A_1$ and that the Chern characters of these two bundles generate $\HH_{0}(\A_1)$. The following questions naturally come to mind:

\begin{quest} \label{questcubic}
\begin{enumerate}
\item Do the objectis $\E_{X_1}(-1)$ and $\E_{X_1}(-2)$ split-generate the category $\A_1$?

\item For $n \geq 0$, denote by $T_{\E_{X_1}(-2)}^{n}$ the $n$-th composition with itself of the spherical twist along $\E_{X_1}(-2)$. Can we reconstruct a smooth elliptic curve, say $E$, and an action of $\mathbb{Z}_3 \times \mathbb{Z}_3$ on $E \times E \times E$ from the ring $\ds \bigoplus_{n \geq 0} \mathrm{Hom}(\E_{X_1}(-1), T_{\E_{X_1}(-2)}^{n}(\E_{X_1}(-1)))$?

\item Is there a smooth cubic seven fold $X_1$ such that the Fukaya category of $Z_1$ is equivalent to the category of $\A_{\infty}$-modules over the $\A_{\infty}$-algebra $\mathrm{RHom}(\E_{X_1}(-1) \oplus \E_{X_1}(-2), \E_{X_1}(-1) \oplus \E_{X_1}(-2))$?
\end{enumerate}
\end{quest}

The first equation is quite natural and has a positive answer in the context of Fukaya categories (see Theorem $1.1$ of \cite{abouzaid}) or Lemma $9.2$ in \cite{Seidel}). Unfortunately, the analogues of such results are not known in algebraic geometry. We state as a question:

\begin{quest} \label{questgenerate}
Let $X$ be a smooth projective variety and $\A$ be a semi-orthogonal component of $\DB(X)$ which is Calabi-Yau of dimension $p \geq 0$. Let $\E_1, \ldots, \E_k$ be $p$-spherical objects in $\A$ whose ranks with respect to the rank function coming from $X$ are non-zero. Assume that the Chern characters of $\E_1, \ldots, \E_k$ generate $\HH_{0}(\A)$. Is it true that the object $\E_1, \ldots, \E_k$ split generate $\A$?
\end{quest}

The assumption that $\A$ is Calabi-Yau can not be withdrawn, owing to the existence of phantom categories (see \cite{orgor} for instance). Nevertheless, the Homological Mirror Symmetry conjecture and the truth of the statement corresponding to question \ref{questgenerate} in the context of Fukaya categories lead us to believe that it should have a positive answer.

\bigskip

\noindent The second item of question \ref{questcubic} is inspired by known proofs of Homological Mirror Symmetry for elliptic curves (see \cite{zaslowpol, zaslow}, see also \cite{poli2018}). Provided that the answer to the first item of question \ref{questcubic} is positive, the third item of question \ref{questcubic} is just a formulation of the Homological Mirror Symmetry conjecture for the cubic sevenfold and the rigid Calabi-Yau threefold $Z_1$.

\begin{rem}
Let $T = E \times E \times E$ be the triple product of an elliptic curve $E$, given by the equation $\{Z_1^3+Z_2^3 + Z_3^3 = 0 \} \subset \mathbb{P}^2 $ and let $\mathbb{Z}_3$ acts on $T$ as follows:
\begin{equation*}
\begin{split}
(1) . (z_1,z_2,z_3,z_4,z_5,z_6,z_7,z_8,z_9) & = (\alpha.z_1, z_2,z_3,\alpha .z_4,z_5,z_6,\alpha.z_7,z_8,z_9) 
\end{split}
\end{equation*} 
where $\alpha$ is a cubic root of unity. The quotient $T/\mathbb{Z}_3$ has crepant resolution which is a Calabi-Yau threefold. We denote it by $Z_2$. The Hodge diamond of $Z_2$ is:

\begin{center}
\begin{tabular}{ccccccc} 
& & & 1 & & & \\
& & 0& &0 & &  \\
& 0& &36 & &0 & \\
1& &0 & &0 & &1 \\
& 0& &36 & &0 & \\
& & 0& &0 & &  \\
& & & 1 & & & \\
\end{tabular}
\end{center}

We consider the cubic sevenfolds having equations of type:
\[X_{\underline{b}} = \{ -z_1z_2z_3 - z_4z_5z_6-z_7z_8z_9 + \sum_{i=1}^{9} b_iz_i^3 +  \sum_{\substack{i \in [1,2,3] \\ j \in [4,5,6] \\ k \in [7,8,9]}} b_{ijk}z_iz_jz_k = 0 \}, \]
where the $\underline{b} = (b_i, b_{ijk})$ is a vector of complex numbers. Let $\mathbb{Z}_3$ acts on $X_{\underline{b}}$ by:

\begin{equation*}
\begin{split}
(1) . (z_1,z_2,z_3,z_4,z_5,z_6,z_7,z_8,z_9) & = (\alpha.z_1, \alpha.z_2,\alpha.z_3,\alpha^2 .z_4,\alpha^2.z_5,\alpha^2.z_6,z_7,z_8,z_9) 
\end{split}
\end{equation*} 
The exists a semi-orthogonal decomposition:
\begin{equation*}
\DB(X_{\underline{b}})^{\mathbb{Z}_3} = \langle \A_{\underline{b}}, \OO_{X_{\underline{b}}}, \ldots, \OO_{X_{\underline{b}}}(5) \rangle,
\end{equation*}
where $\DB(X_{\underline{b}})^{\mathbb{Z}_3}$ is the derived category of $\mathbb{Z}_3$-equivariant coherent sheaves on $X_{\underline{b}}$. In case $X_{\underline{b}}$ is smooth, the category $\A_{\underline{b}}$ is Calabi-Yau of dimension $3$. It is announces in \cite{sherismith} that (for the right choice of $\underline{b}$) there is an equivalence between $\A_{\underline{b}}$ and a version of the Fukaya category of $Z_2$. It would be interesting to check if their techniques can be applied to answer the third item of question \ref{questcubic}.
\end{rem}

\end{subsubsection}

\begin{subsubsection}{Double quartic fivefold}
The story for the double quartic fivefold is very similar to that of the cubic sevenfold. Namely, let $T = E \times E \times E$ be the triple product of an elliptic curve $E$, given by the equation $\{Z_1^4+Z_2^4 + Z_3^2  = 0 \} \subset \mathbb{P}(1,1,2) $ and let $\mathbb{Z}_4 \times \mathbb{Z}_4$ acts on $T$ as follows:

\begin{equation*}
\begin{split}
(1,0) . (z_1,z_2,z_3,z_4,z_5,z_6,z_7,z_8,z_9,) & = (-i.z_1, z_2,-z_3,z_4,z_5,z_6,i.z_7,z_8,-z_9 \\
(0,1) . (z_1,z_2,z_3,z_4,z_5,z_6,z_7,z_8,z_9) & = (z_1,z_2,z_3,- i.z_4,z_5,-z_6, i.z_7,z_8,-z_9),
\end{split}
\end{equation*} 
where $i$ is a square root of $-1$. The quotient $T/\mathbb{Z}_4 \times \mathbb{Z}_4$ has crepant resolution which is a Calabi-Yau threefold. We denote it by $Z_3$. The Hodge diamond of $Z_3$ is:

\begin{center}
\begin{tabular}{ccccccc} 
& & & 1 & & & \\
& & 0& &0 & &  \\
& 0& &90 & &0 & \\
1& &0 & &0 & &1 \\
& 0& &90 & &0 & \\
& & 0& &0 & &  \\
& & & 1 & & & \\
\end{tabular}
\end{center}

As explained in \cite{CDP, sch, BBVW}, the mirror of $Z_3$ ought to be a Landau-Ginzburg model related to a smooth double quartic fivefold. Let $X_3 \subset \mathbb{P}(1,1,1,1,1,1,2)$ be a smooth double quartic fivefold. We have a semi-orthogonal decomposition:

\begin{equation*}
\DB(X_3) = \langle \A_3, \OO_{X_1}, \ldots, \OO_{X_1}(3) \rangle,
\end{equation*} 
where $\A_{3}$ is a Calabi-Yau category of dimension $3$. It follows from \cite{orlov1} that $\A_3$ is the homotopy category of the $DG$-category of graded matrix factorizations of the equation of $X_1$. Therefore, the category $\A_3$ can be interpreted as a Landau-Ginzburg model for the double quartic fivefold $X_3$. We found out in example \ref{exemhodgenumbers} that the Hodge diamond of $\A_3$ is:

\begin{center}
\begin{tabular}{ccccccc} 
& & & 1 & & & \\
& & 0& &0 & &  \\
& 0& &0 & &0 & \\
1& &90 & &90 & &1 \\
& 0& &0 & &0 & \\
& & 0& &0 & &  \\
& & & 1 & & & \\
\end{tabular}
\end{center}
which is again a favorable presage as far as mirror symmetry is concerned. 

\bigskip

We have shown in example \ref{exem1} that there exists a $3$-spherical bundle $\F_{X_3}$ on $X_3$ such that $\F_{X_3}(-1)$ and $\F_{X_3}(-2)$ are in $\A_3$ and that the Chern characters of these two bundles generate $\HH_{0}(\A_3)$. We ask for $X_3$ the analogous question to \ref{questcubic}:

\begin{quest} \label{questquartic}
\begin{enumerate}
\item Do the objectis $\F_{X_3}(-1)$ and $\F_{X_3}(-2)$ split-generate the category $\A_3$?

\item For $n \geq 0$, denote by $T_{\F_{X_3}(-2)}^{n}$ the $n$-th composition with itself of the spherical twist along $\F_{X_3}(-2)$. Can we reconstruct a smooth elliptic curve, say $E$, and an action of $\mathbb{Z}_4 \times \mathbb{Z}_4$ on $E \times E \times E$ from the ring $\ds \bigoplus_{n \geq 0} \mathrm{Hom}(\F_{X_3}(-1), T_{\F_{X_3}(-2)}^{n}(\F_{X_3}(-1)))$?

\item Is there a smooth double quartic fivefold $X_3$ such that the Fukaya category of $Z_3$ is equivalent to the category of $\A_{\infty}$-modules over the $\A_{\infty}$-algebra $\mathrm{RHom}(\F_{X_3}(-1) \oplus \F_{X_3}(-2), \F_{X_3}(-1) \oplus \F_{X_1}(-2))$?
\end{enumerate}
\end{quest}
It would be again interesting to know if the techniques developed in \cite{sherismith} could be used to answer the third item of question \ref{questquartic}.

 \end{subsubsection}
\end{subsection}
\end{section}

\newpage

\bibliographystyle{alpha}
\bibliography{HKR}
\newpage
\end{document}